\documentclass[pdflatex,sn-mathphys-num]{sn-jnl}

\usepackage{graphicx}
\usepackage{multirow}
\usepackage{amsmath,amssymb,amsfonts}
\usepackage{amsthm}
\usepackage{mathrsfs}
\usepackage[title]{appendix}
\usepackage{xcolor}
\usepackage{textcomp}
\usepackage{manyfoot}
\usepackage{booktabs}
\usepackage{algorithm}
\usepackage{algorithmicx}
\usepackage{algpseudocode}
\usepackage{listings}

\usepackage[english]{babel}
\usepackage{amssymb}
\usepackage{amsmath}
\usepackage{mathtools}
\usepackage{epsfig}
\usepackage{psfrag}
\usepackage{amsthm}
\usepackage{graphicx}
\usepackage{subfigure}
\usepackage{hyperref}
\definecolor{darkblue}{rgb}{0.0, 0.0, 0.8}
\hypersetup{
    colorlinks = true,
    citecolor = {darkblue},
    linkcolor = {darkblue}
}

\newtheorem{lemma}{Lemma}
\newtheorem{theorem}{Theorem}
\newtheorem{proposition}{Proposition}
\newtheorem{definition}{Definition}
\newtheorem{remark}{Remark}
\newtheorem{corollary}{Corollary}
\newtheorem{assumption}{Assumption}

\newcommand{\be}{\begin{equation}}
\newcommand{\ee}{\end{equation}}

\newcommand{\E}{\mathbb E}

\newcommand{\dd}{\mathrm{d}}

\renewcommand{\P}{P}
\newcommand{\Ps}{{P^\ast}}
\newcommand{\tb}[1]{{{#1}}}

\newcommand{\vertiii}[1]{{\left\vert\kern-0.25ex\left\vert\kern-0.25ex\left\vert #1 
    \right\vert\kern-0.25ex\right\vert\kern-0.25ex\right\vert}}

\definecolor{amber(sae/ece)}{rgb}{1.0, 0.49, 0.0}
\definecolor{amethyst}{rgb}{0.6, 0.4, 0.8}
\definecolor{antiquefuchsia}{rgb}{0.57, 0.36, 0.51}
\definecolor{auburn}{rgb}{0.43, 0.21, 0.1}
\definecolor{cadmiumgreen}{rgb}{0.0, 0.42, 0.24}
\definecolor{byzantium}{rgb}{0.44, 0.16, 0.39}
\definecolor{burntsienna}{rgb}{0.91, 0.45, 0.32}

\newcommand{\bgcmt}[1]{{\color{burntsienna}}}

\newcommand{\ms}[1]{{\medskip\noindent #1}}

\fussy
\begin{document}

\title[Stochastic Approximation in a Markovian Framework]{\sc \LARGE Stochastic Approximation in a Markovian Framework Revisited:
Lipschitz Continuity of the Poisson Equation\vspace*{5mm}}

\author[2]{\fnm{Algo} \sur{Car\`e}}\email{algo.care@unibs.it}

\author*[1,4]{\fnm{Bal\'azs Csan\'ad} \sur{Cs\'aji}}\email{csaji.balazs@sztaki.hun-ren.hu}

\author[3,4]{\fnm{Bal\'azs} \sur{Gerencs\'er}}\email{gerencser.balazs@renyi.hun-ren.hu}

\author[1]{\fnm{L\'aszl\'o} \sur{Gerencs\'er}}\email{gerencser.laszlo@sztaki.hun-ren.hu}

\author[3,4]{\fnm{Mikl\'os} \sur{R\'asonyi}}\email{rasonyi.miklos@renyi.hun-ren.hu}

\affil*[1]{\orgdiv{{HUN-REN} Institute for Computer Science and Control}, \orgname{Hungarian Research Network}, \orgaddress{\street{Kende u.\ 13-17.}, \city{Budapest}, \postcode{H-1111}, \country{Hungary}}}

\affil[2]{\orgdiv{Dipartimento di Ingegneria dell'Informazione}, \orgname{Universit\`a di Brescia}, \orgaddress{\street{via Branze 38}, \city{Brescia}, \postcode{25123}, \country{Italy}}}

\affil[3]{\orgdiv{{HUN-REN} Alfr\'ed R\'enyi Institute of Mathematics}, \orgname{Hungarian Research Network}, \orgaddress{\street{Reáltanoda u.\ 13-15.}, \city{Budapest}, \postcode{H-1053}, \country{Hungary}}}

\affil[4]{\orgdiv{Institute of Mathematics}, \orgname{E\"otv\"os Loránd University (ELTE)}, \orgaddress{\street{P\'azm\'any P\'eter s\'et\'any 1/C}, \city{Budapest}, \postcode{H-1117}, \country{Hungary}}}

\abstract{In this paper we revisit a fundamental technical issue within the theory of stochastic approximation (SA) in a Markovian framework, first proposed in the book by Djereveckii and Fradkov (1981), and further developed in much detail in the book by Benveniste, M{\'e}tivier, and Priouret (1990). This theory is instrumental in many application areas such as the statistical analysis of Hidden Markov Models arising in  telecommunication, quantized linear stochastic systems, and more recently in active learning and reinforcement learning. The problem at hand is the verification of the existence, uniqueness and Lipschitz-continuity of the solution of a parameter-dependent Poisson equation, in an appropriate weighted sup-norm, associated with a collection of Markov chains on general state spaces. Verification of the above facts is vital in the analysis of SA processes presented in \cite{benveniste1990} via the ODE (ordinary differential equations) method, requiring substantial technical effort. The motivation and focus of the paper is to address this technical issue, by  presenting a simple set of conditions, under which the above properties of the Poisson equation at hand can be conveniently established. The starting point of our work is an intricate result of Hairer and Mattingly (2011) proving that by tilting standard conditions of mainstream stability theory for Markov chains, the transition kernels prove to be contractions in the space of differences of probability measures in a suitable metric. To demonstrate the applicability of our results, the proposed conditions are verified for a class of queuing system with open-loop control. 
}

\maketitle
\newpage

\section{Introduction}
Stochastic approximation (SA) is a fundamental methodology for real-time statistical analysis in important application areas such as signal processing, control systems, and more recently in machine learning. Specifically many algorithms in adaptive filtering, recursive system identification, adaptive methods in input design, adaptive control or reinforcement learning rely on ideas of classical stochastic approximation theory initiated by the celebrated paper of Robbins and Monro \cite{RobbinsMonro1951} back in 1951. 

Stochastic approximation in a Markovian framework, first proposed in \cite{djereveckii1974application,djereveckii1981applied}, and extensively developed in the book \cite{benveniste1990}, was a significant contribution to the area, allowing the construction and analysis of statistical estimation methods for wide class of nonlinear systems, 
such as Hidden Markov Models (HMM-s), arising in telecommunication, or quantized linear stochastic systems. In machine learning, studying SA in a Markovian framework is instrumental for reinforcement learning (RL), such as TD-learning \cite{watkins1992q} and Q-learning \cite{sutton1988learning}. The theory presented in \cite{benveniste1990} is in a sense complementary to the widely used theory for recursive identification of linear stochastic systems, developed by Ljung in \cite{ljung1983theory}, with focus on mixing properties of the driving noise. 

The {\it contribution} of this paper is a significant addition to the theory of SA in a Markovian framework, developed in \cite{benveniste1990}, by introducing much simpler conditions for the Markov chain under which a key technical issue, required for the ODE analysis, can be resolved. {With all technicalities explained in sufficient details the paper is self-contained.} \footnote{ This paper is a significantly extended version of our paper published in the Proceedings of the 
58th IEEE Conference on Decision and Control \cite{PoissonCDC2019}.}

To provide the context of the present paper we briefly describe a few  
{technical aspects central in \cite{benveniste1990}.}
Following their predecessors, the authors of \cite{benveniste1990} present a model that boils down to the solution of a nonlinear algebraic equation, specifically defined in terms of a strictly stationary parameter-dependent Markov process $(X_n(\theta)),$ representing physical signals, their filtered values or a combination of these. 
The process may take its values in an abstract measurable space $\mathbf X$, such as a Euclidean space $\mathbb R^m$, a Hilbert space or a finite discrete set. The parameter $\theta \in \Theta$ may characterize the open-loop system dynamics, the effect of a controller, or the tentative value of the true parameter within system identification. In the context of \cite{benveniste1990}, as in the theory of recursive system identification and adaptive control {of} linear stochastic systems, see \cite{ljung1983theory}, $\Theta$ is typically a subset of a Euclidean space, $\mathbb R^k.$ We should note that in some recent machine learning applications,
such as parameter-free online learning \cite{cutkosky2018black} and stochastic approximation for kernel methods \cite{JMLR:v25:23-0168}, $\Theta$ is a subset of a 
(typically infinite dimensional) Banach or Hilbert space. Finally, studying ODE-based stochastic approximation in a Markovian framework is also crucial for RL methods \cite{JMLR:v26:24-0100}.

The dynamics of the Markov process $(X_n(\theta))$ would be classically described by its transition probability kernels $\P_\theta(x,A)$, with $x \in \mathbf X, \, A \subset  \mathbf X$, $A$ being measurable, and indeed, we shall follow suit in the rest of the paper. But for now it is preferable to take a system's point of view, and define the process 
explicitly via
\begin{equation}
X_{n+1}(\theta)   = F(\theta ; X_n(\theta), W_{n+1}),
\end{equation}
where $F$ is a measurable mapping, and $(W_{n})$ is a sequence of i.i.d. (independent, identically distributed) random variables. The sequence $(W_{n})$ may represent exogenous system noise, measurement noise, or a dither injected by the user. The objective is to identify or to tune the parameter so that some appropriately defined asymptotic cost function expressing reconstruction error or tracking error is minimized. The instantaneous cost {is}  a function of both $\theta$ and $x \in \mathbf X.$

Assuming $\theta \in \mathbb R^k$, an $\mathbb R^k$-valued pseudo-gradient of the {(instantaneus) cost function} w.r.t. (with respect to) $\theta$ will be defined, specific for the problem at hand, denoted by $H(\theta ; x).$ Then the
estimation problem reduces to 
solving the non-linear algebraic equation, for any $n$, 
in view of stationarity of $X_n(\theta),$
\begin{equation}
\label{eq:BMP basic-equation}
\E\hspace{0.1mm}\tb{\big[} \hspace{0.1mm} H(\theta ; X_n(\theta)) \hspace{0.1mm} \tb{\big]}\hspace{0.2mm}=\, 0.
\end{equation}
The practical objective is to find the root of \eqref{eq:BMP basic-equation}, denoted by $\theta^\ast,$ via a recursive algorithm based on computable approximations of $H(\theta ; X_n(\theta)).$ 

Computability means that the r.h.s.\ (right hand side) of \eqref{eq:SA X}-\eqref{eq:SA theta} below can be evaluated via the cyber-physical system at hand. 
The proposed SA algorithm of \cite{benveniste1990} is 
\begin{align}
\label{eq:SA X}
X_{n+1} &= F (\theta_n ; X_{n}, W_{n+1}), \\
\label{eq:SA theta}
\theta_{n+1} &= \theta_{n} + {\frac 1 {n+1}} H (\theta_n ; X_{n}).
\end{align}
We note in passing that computability is critical in  applications such as stochastic adaptive control or adaptive input design, see \cite{gerencser2024notes}.

An early version of the above problem and the associated algorithm is presented in \cite{ljung1977analysis}, in which $(X_n(\theta))$ is assumed to be defined via a linear stochastic system driven by a weakly dependent process, exhibiting certain weak forms of stationarity.

The parameter estimate is often forced to stay in a compact domain using resetting, see \cite{benveniste1990} or for a more recent paper \cite{Borkar_et_al2023}. For the convergence analysis of the above algorithm a standard approach is the ODE (ordinary differential equation) method, see  \cite{ljung1983theory, benveniste1990, kushner2003applications}, in which the discrete sequence $(\theta_n)$ is approximated by the solution trajectory $y_t$ of the ODE, for any $m$,  
\begin{equation}
\label{eq:The ODE}
\dot y_t = {\frac 1 t}  \mathbb E \, [\, H (y_t ; X_m(y_t))\,]. 
\end{equation}

If the ODE has an asymptotically stable equilibrium point at $\theta^\ast$ with a reasonable domain of attraction then we may expect that $y_t$ will track $\theta_n$ at $t=n.$ In order to capture the tracking error $|\theta_n-y_n|$ on finite intervals,  following a sequence of intricate arguments, see \cite[Part II, Chapters 1 and 2]{benveniste1990}, we arrive at the problem of estimating the additive functional 
\vspace{-1mm} 
\begin{equation}
\label{eq:Additive Functional on Markov}
\sum_{n=1}^{N} \hspace{-0.3mm}\left(H(\theta ; X_n(\theta)) - \E_{\mu_\theta}\hspace{-0.2mm} \tb{\big[ } H(\theta ; X_n(\theta)) \tb{\big] } \right)\hspace{-0.4mm},
\end{equation}
where $\mu_\theta$ is the assumed unique invariant measure of $X_n(\theta)$ under $\theta$.
A well-known device in the theory of Markov processes is to express  \eqref{eq:Additive Functional on Markov} using a Markovian version of the Newton-Leibniz formula of basic calculus by representing the individual terms via the solution of the Poisson equation 
\begin{equation}
(I-\P_\theta^\ast)\hspace{0.5mm} u_\theta(x) \,=\,  H (\theta ; x) - 
\E_{\mu_\theta}\hspace{-0.2mm} \tb{\big[ } H(\theta ; X(\theta)) \tb{\big] }. 
\end{equation}
Here $u_\theta(x)$ is an unknown function, playing the role of a primitive function, and $\P_\theta^\ast $, the adjoint of $\P_\theta,$ is given by \eqref{eq:Pstar} below. Thus \eqref{eq:Additive Functional on Markov} will become the sum of martingale differences. 
For a historical perspective on the topic see the early paper   
\cite{schweitzer1968perturbation}.

In the ODE analysis proposed in \cite[Part II, Chapter 2]{benveniste1990} a vital technical tool is the verification of the Lipschitz continuity of $u_\theta(x)$ w.r.t.\ $\theta$. This requires substantial technical efforts, and the conditions under which useful results (such as \cite[Part II, Chapter 2, Theorem 6]{benveniste1990}) are derived are quite demanding, see the final note below at the end of Section \ref{sec:Lipschitz_Continuity_of_the_Kernel}. 

The motivation and focus 
of the present paper is to provide a significantly simpler set of conditions, with Assumption \ref{cond:Lipschitz of P theta Dirac x by Miklos} below playing a central role, under which the existence, uniqueness and Lipschitz continuity of the solution of parameter-dependent Poisson equations can be conveniently established.

The methodology of our investigation is based on a powerful result in the stability theory for Markov chains, developed in \cite{hairer2011yet}. A major observation of \cite{hairer2011yet} is that by tilting standard conditions of mainstream stability theory for Markov chains, see \cite{meyn2012markov}, the transition kernels prove to be contractions in the space of {differences of probability measures in a suitable metric, see Proposition \ref{prop:Contraction of sigma beta}.}
We briefly present and interpret the main results of \cite{hairer2011yet} in a self-contained manner in Section \ref{sec:Introduction to HM}.
Application of this mathematical technology to the subject matter of the paper results in a transparent and flexible analysis. 

The prospective advantage of using the methodology of Hairer and Mattingly, \cite{hairer2011yet} is that (uniform) contractivity of the probability transition kernels for $\theta \in \Theta$ implies the stability of the inhomogeneous Markov-chain given by \eqref{eq:SA X} with arbitrary $\theta_n \in \Theta.$ Ensuring some kind of stability of \eqref{eq:SA X} is in fact a key issue in SA, see \cite{benveniste1990}.

The recent interest in stochastic approximation in a Markovian framework is also reflected in \cite{JMLR:v26:24-0100} with a focus on the ODE method. However, in contrast to the general theory of \cite{benveniste1990} the Markov chain, driving the SA iteration, does not depend on the parameter. Therefore,
it bypasses the particular technical issue settled in this paper.  

The same limited scope has been chosen in the early versions of \cite{Borkar_et_al_arXiv_2021_2024}. In this latest version, representing a major extension of the results in prior versions,  which has since been recently published \cite{Borkar_et_al2023}, the authors have considered SA processes driven by Markov chains depending on a parameter $\theta \in \mathbf R^d$, partially following the logic of \cite{benveniste1990}. In particular, in analogy with \cite{PoissonCDC2019}, they proved Lipschitz continuity of the solution of parameter-dependent Poisson equations within the context of mainstream stability theory for Markov chains \cite{meyn2012markov}, see their Proposition 7. However, geometric ergodicity and its variants do not imply contractivity of the probability transition kernels in any of the metrics used in  \cite{meyn2012markov}. Hence proving stability of the inhomogeneous Markov-chain given by \eqref{eq:SA X} with arbitrary $\theta_n \in \mathbf R^d,$ without further enforced restrictions on the dynamics of $\theta_n,$ as in \cite{benveniste1990}, requires substantial technical efforts.

The structure of the paper is as follows: in Section \ref{sec:Introduction to HM} we provide a brief introduction to the stability theory for Markov chains developed in \cite{hairer2011yet}, with appropriate interpretations and eventual simplifications. In Section \ref{sec:Poisson Equation - Existence UniquenessMain results} fundamental properties of the Poisson equation, existence and uniqueness of solutions, are discussed. In Section \ref{sec:Lipschitz_Continuity_of_the_Kernel} a simple condition imposing the Lipschitz continuity of the kernel is introduced, and its implications are discussed. The first main result of the paper is stated in Section \ref{sec:Poisson Equation - Lipschitz Continuity}, as Theorem \ref{thm:Poisson Lipschitz}, stating  the Lipschitz continuity of the solutions of a parameter-dependent Poisson equation under reasonable conditions.

In Section \ref{sec:Relaxations of the uniform drift condition} the cited results of \cite{hairer2011yet} are restated under relaxed conditions, 
in particular, imposing conditions on some power of the kernel, $\P_\theta^r,$ reflecting the demarcation between contractivity and stability of ordinary matrices.
In Section \ref{sec:Poisson Equations under Relaxed Conditions} we restate the results of  Section \ref{sec:Poisson Equation - Lipschitz Continuity} under relaxed conditions, leading to our second main result, Theorem \ref{th:Poisson Lipschitz w Uniform cond on Pr}. In Section \ref{sec:Design of Quees} the viability of our results is demonstrated on the modification of a classical textbook example, the design of a simple queuing system with open-loop control, a preliminary version of which has been presented in \cite{PoissonCDC2019}.
The paper is concluded with a brief discussion on potential future research directions.

Given the technical nature of the paper, pre-determined by the subject matter,
we have chosen a semi-classical structure to enhance readability: relevant concepts, novel theorems and their structured proofs are presented in the main body of the paper, whereas the details of the proofs of lemmas and corollaries are given in the Appendix.

{
\section{On a Theorem of Hairer and Mattingly}
\label{sec:Introduction to HM}
}

{In this section we provide a brief summary of an intricate addition to mainstream stability theory for Markov Chains, developed in \cite{hairer2011yet}.} Consider a family of Markov chains $(X_n(\theta)), ~ 0 \le  n < + \infty$ with arbitrary state space $\mathbf X$ equipped with a $\sigma$-field $\mathcal{A}$ of events, and $\Theta$ being an open set of a normed space. Therefore, we consider the $\theta$-dependent transition probability kernels $\P_\theta(x,A)$, with $x \in \mathbf X, \, A \in \cal A$, a shorthand for the conditional probability $P(X_1(\theta) \in A \!\mid\! X_0(\theta)=x)$. We will assume that 
for each $A \in \mathcal{A}$, $\P_\cdot(\cdot,A)$ is $(x,\theta)$-measurable. These assumptions can always be satisfied if $\mathbf X$ is a Polish space and $\Theta$ is a subset of a separable Banach space. 

\medskip\noindent
For any probability measure $\mu$ over $(\mathbf X, \cal A)$ and measurable 
$\varphi : {\mathbf X} \rightarrow \mathbb{R}$ 
define
\begin{equation}
\label{eq:Pstar}
\begin{aligned}
(\P_\theta \mu)(A) &\!:=\!\!\int_{\mathbf X}\!P_\theta(x,A)\mu(\mathrm{d}x)  , \\[1mm]
(P_\theta^\ast \varphi)(x) &\!:=\!\!\int_{\mathbf X}\!\varphi(y) \P_\theta(x,\mathrm{d}y) = \E_{\theta} \big[\, \varphi(X_{1}) \mid X_0 = x\,\big],
\end{aligned}
\end{equation}
assuming the second integral exists. 
Hence, $(\P_\theta \mu)(.)$ is the probability measure of $X_1(\theta)$ assuming that the probability measure of $X_0(\theta)$ is $\mu,$ and $(P_\theta^\ast \varphi)(x)$ is the conditional expectation  $\E_{\theta} \big[\, \varphi(X_{1}) \mid X_0  = x\,\big].$ The next condition is motivated by \cite{hairer2011yet}, stated there for a single Markov-chain.

\begin{assumption}[Uniform Drift Condition for $P_{\theta}$]
\label{cond:Uniform Drift Condition}
There exists a measurable function $V:\mathbf{X}\rightarrow [0,\infty)$ and constants $\gamma\in(0,1)$ and $K\geq0$ such that for all $x\in{\mathbf{X}}$ and $\theta\in\Theta$
\begin{equation}
(P_\theta^\ast V)(x)\,\leq\, \gamma V(x)+K.
\label{eq:DriftCondition}
\end{equation}
\end{assumption}

$V(x)$ is called a Lyapunov function. Note that $V(x), \gamma$ and $K$ are not $\theta$-dependent. Let us take a measure $\mu$ such that 
\begin{equation}
\label{eq:mu V mu X less infty}
 \mu(V) := \int_{\mathbf X} V(x) \mu(\mathrm{d}x) <\infty \quad {\rm and} \quad \mu(\mathbf X) < \infty.   
\end{equation}
Then, integrating \eqref{eq:DriftCondition} with respect to $\mu$ we get for all $\theta \in \Theta:$
\begin{equation}
\P_{\theta} \mu (V)\, \leq\, \gamma \mu (V) + K \mu(\mathbf X).
\label{eq:Uniform Drift Condition w mu}
\end{equation}
Condition \eqref{eq:mu V mu X less infty} is often expressed as $\mu(1+\beta V) < \infty$ with some (and therefore any) $\beta >0.$ Inequality \eqref{eq:Uniform Drift Condition w mu} extends for any signed measure $\eta,$ with $|\eta|(1+\beta V) < \infty$:  
\begin{equation}
|\P_{\theta} \eta| (V)\, \leq\, \gamma |\eta| (V) + K |\eta|(\mathbf X),
\label{eq:Uniform Drift Condition w eta}
\end{equation}
for all $\theta \in \Theta,$ due to the inequality $|P_{\theta}\eta| \leq P_{\theta} |\eta|$. The set of signed measures $\eta$ with $|\eta|(1+\beta V) < \infty$ is denoted by ${\cal M}_{V}.$

Note that the Lyapunov function $V(\cdot)$ can be fairly general, as opposed to \cite{benveniste1990}. In particular, we may use  
$V(x) = e^{cx},$ known to be the right choice for 
queuing systems \cite[Section 16.4]{meyn2012markov}.
The next condition is a natural extension of Assumption 2 of \cite{hairer2011yet} for a parametric family of Markov chains, which itself is a modification of a standard condition in the stability theory of Markov chains \cite{meyn2012markov} requiring minorization on what is called a small set.

\begin{assumption}[Local Minorization]
\label{cond:Local Minorization}
Let $R > 2K/(1\!-\!\gamma)$, where  $\gamma$ and $K$ are the constants from Assumption \ref{cond:Uniform Drift Condition}, and define the set ${\cal C}=\{x\in{\mathbf X} :  V(x) \leq R \}$. There exist a probability measure $\bar{\mu}$ on $\mathbf X$ and a constant $\bar{\alpha}\in(0,1)$ such that, for all $\theta\in\Theta$, all $x\in{\cal C}$, and all measurable $A$,
\begin{equation}
\P_\theta(x,A)\,\geq\,\bar{\alpha}\bar{\mu}(A).
\end{equation}
\end{assumption}

\begin{remark}[Interpretation of $R$]
\label{rem:int V wrt mu star}
If there exists an invariant measure $\mu^\ast_\theta$ such that $\int_{\mathbf X} V(x) \mu^\ast_\theta ( \mathrm{d}x)<\infty,$ then integrating both sides of inequality \eqref{eq:DriftCondition}, we infer that  
\begin{equation}
\int_{\mathbf X} V(x) \mu^\ast_\theta ( \mathrm{d}x)\,\leq\, \frac{K}{1-\gamma}.
\label{eq:VdmuBound}
\end{equation}
Thus, the parameter $R$ in Assumption \ref{cond:Local Minorization} must exceed twice the mean of $V$ w.r.t.\ any of the invariant measures.
\end{remark}

\begin{remark}[Constant shifts of $V$] 
\label{rem:Constant Shifts of V} 
We can and will assume that $\inf_x V(x)=0$ without loss of generality. In fact, if a function $V$ satisfies  Assumptions \ref{cond:Uniform Drift Condition} and \ref{cond:Local Minorization}, then  $V' := V - \inf_x V(x)$ also satisfies  Assumptions \ref{cond:Uniform Drift Condition} and \ref{cond:Local Minorization} with the same contraction coefficient $\gamma,$ the same minorization domain $\mathcal{C}$ and the same $\overline \alpha$, with appropriately chosen parameters $K^\prime$ and $R^\prime.$
\end{remark}

A key technique of \cite{hairer2011yet} used for the stability analysis of Markov processes is the use of what is called the weighted total variation distance between two probability measures:  
\begin{definition}
\label{def:rhoBetaAsATotalVariation}
Let $\mu_1$ and $\mu_2$ be two probability measures on $\mathbf X$ and $\beta >0$. Define the weighted total variation distance as
\begin{equation}
\rho_\beta(\mu_1,\mu_2):=\int_{\mathbf X} (1+\beta V(x))|\mu_1-\mu_2|(\mathrm{d}x),
\end{equation}
where $|\mu_1-\mu_2|$ is the total variation measure of $(\mu_1-\mu_2)$.
\end{definition}

For $\beta  = 0$ the weighted total variation distance would become just the standard total variation distance $|\mu_1-\mu_2|_{\rm TV}.$ 

Alternatively, writing $\eta = \mu_1 - \mu_2,$ we can define 
\begin{equation}
\label{def:rhobetaeta}
\rho_\beta(\eta):=\int_{\mathbf X} (1+\beta V(x))|\eta|(\mathrm{d}x).
\end{equation}

Note that the above definition extends to any signed measure $\eta$ such that  $\int_{\mathbf X} (1+\beta V(x)) |\eta| (\mathrm{d}x)  < \infty,$ leading to what will be called the weighted total variation norm of $\eta.$

An equivalent definition of $\rho_\beta$ can be given by allowing more general weighting functions $\varphi$ : ${\mathbf X} \rightarrow \mathbb{R}$ replacing ${1+\beta  V(x)}.$ To be more specific, let us introduce the norm:

\begin{definition} For any measurable function $\varphi$ : ${\mathbf X} \rightarrow \mathbb{R}$, set 
\label{def:DoubleNorm}
\begin{equation}
\|\varphi\|_\beta:=\sup_{x}\frac{|\varphi(x)|}{1+\beta  V(x)}.
\label{eq:DoubleNorm}
\end{equation}
\end{definition}

The linear space spanned by the functions such that $\|\varphi\|_\beta < \infty$ will be denoted by ${\cal L}_{V}$. Note that  ${\cal L}_{V}$ is neither affected by constant shift of $V$, nor the choice of $\beta$; moreover, ${\cal L}_V$ with the norm $\|\cdot\|_\beta$  becomes a Banach space for any $\beta>0.$ 

In possession of the norm $\|\cdot\|_\beta$ an equivalent definition of the weighted total variation norm can be obtained as follows:

\begin{equation}
\rho_\beta(\eta) := \sup_{\varphi:\|\varphi\|_\beta\leq 1} \int_{\mathbf X} \varphi(x)\eta (\mathrm{d}x).
\label{eq:RhoBetaDefwithDoubleNormSingleVar}
\end{equation}

Weighted total variation norms have been also used in the classic book \cite{meyn2012markov}, introduced in Chapter 14. Conditions for geometric convergence of $\P^n(x,.)$ to the unique invariant measure, interpreted via weighted total variation norms, are given in Theorem 16.0.1, see Remark \ref{rem:HM vs MT} for details. 
However, the smart choice of the weighting factor $\beta,$ ensuring the contractivity of $P_\theta^\ast$, showing up in Proposition \ref{prop:Beta Norm Contraction by P}, and the simplicity of the conditions of \cite{hairer2011yet} are new.

To capture the smoothing effect of $P_\theta^\ast$ acting on ${\cal L}_V,$  define a measure of oscillation for functions $\varphi \in {\cal L}_V $ as follows:
\begin{definition} 
\label{def:Triple Norm}
For any function $\varphi \in {\cal L}_V $ set
\begin{equation}
\label{eq:Triple Norm}
\vertiii{\varphi}_\beta=\min_{c\in\mathbb{R}}\|\varphi+c\|_\beta.
\end{equation}
\end{definition}

It is readily seen that $\|\varphi+c\|_\beta$ is continuous in $c$, and unbounded when $c$ tends to $\pm \infty,$ hence the right hand side of \eqref{eq:Triple Norm} is well-defined. Obviously, $\vertiii{\varphi}_\beta \leq \|\varphi\|_\beta$. 

It is easily seen that $ \vertiii{\cdot}_\beta$ is a semi-norm on ${\cal L}_V$ and $ \vertiii{\varphi}_\beta = 0$ if and only if $\varphi$ is a constant function. Letting ${\mathbb R}_X$ denote the linear vector-space of constant functions on $\mathbf X$ it follows that 
	$ \vertiii{\cdot}_\beta$ is a norm on the linear factor-space ${\cal L}_{V,0} : = {\cal L}_V/{\mathbb R}_X.$ It is also easily seen that  ${\cal L}_{V,0}$ becomes a Banach space with the norm $ \vertiii{\cdot}_\beta.$ In what follows, ${\cal L}_{V,0}$ will denote the latter Banach space with a specific, fixed  choice of $\beta$ to be described in Proposition \ref{prop:Beta Norm Contraction by P}. We note that the above definition of $\vertiii{\varphi}_\beta$ is a 
    simplification of what is given in \cite{hairer2011yet}.

A useful linear subspace of the dual space ${\cal L}_{V,0}^\ast$ is obtained by considering the linear space of signed measures $\eta$ such that 
\begin{equation}
\int_{\mathbf X} (1+\beta V(x)) |\eta|(\mathrm{d}x) < \infty, \quad\text{and}\quad \eta(\mathbf X) = 0,
\end{equation}
which will be denoted by ${\cal M}_{V}^0.$ It is easily seen that
\begin{equation}
\varphi\;\; \tb{\xmapsto{\;\;\;}}\, \int_{\mathbf X}\varphi(x) \eta(\mathrm{d}x) 	
\end{equation}
is a continuous linear functional the dual norm of which is 
\begin{equation}
\sigma_\beta(\eta):=\sup_{\varphi:\vertiii{\varphi}_\beta\leq 1} \int_{\mathbf X} \varphi(x)\eta (\mathrm{d}x).
\label{eq:sigma beta def with eta}
\end{equation}

\begin{proposition} [cf.\ \cite{hairer2011yet}] 
\label{prop:sigma beta eq rho beta}
For any $\eta \in {\cal M}_{V}^0$ and any $\beta>0$
\begin{equation}
\sigma_\beta(\eta) = \rho_\beta(\eta).
\label{eq:sigma beta eq rho beta}
\end{equation}
\end{proposition}
The dual approach in defining the same norm proved to be and will prove to be extremely useful. 

For the sake of clarity we briefly recapitulate the 
argument leading to \eqref{eq:sigma beta eq rho beta}.  Obviously, the set $\{\varphi: \|\varphi\|_\beta\leq 1\}$ is a subset of $\{\varphi: \vertiii{\varphi}_\beta \leq 1\},$ hence for any signed measure $\eta  \in \mathcal{M}_V^0$:
\vspace{-1mm}
\begin{align}
\sup_{\varphi:\vertiii{\varphi}_\beta\leq 1} \int_{\mathbf X} \varphi(x)\eta(\mathrm{d}x) 
\ge \sup_{\varphi:\|\varphi\|_\beta\leq 1} \int_{\mathbf X} \varphi(x)\eta(\mathrm{d}x). 
\label{eq:sigma beta eq rho beta proof}
\end{align}

On the other hand, for any fixed $\varphi$ such that $\vertiii{\varphi}_\beta\leq 1$ there exists a $c$ such that $\|\varphi + c\|_\beta\leq 1.$ But 
\begin{align}
\int_{\mathbf X} \varphi(x)\eta(\mathrm{d}x) = 
\int_{\mathbf X} (\varphi(x) + c) \eta(\mathrm{d}x). 
\end{align}

\noindent
Hence a strict inequality in \eqref{eq:sigma beta eq rho beta proof} cannot occur. \qed

The norm $\sigma_\beta(.)$ defined for signed measures $\eta \in {\cal M}_{V}^0$ naturally leads to the following definition of a metric:

\begin{definition} 
\label{def:sigma beta for pairs} Let $\mu_1, \mu_2$ be two possibly signed measures on $\mathbf X$ such that $\int_{\mathbf X} (1+\beta V(x)) |\mu_i|(\mathrm{d}x) < \infty$ for $i=1,2,$ moreover $\mu_1(\mathbf X) = \mu_2(\mathbf X).$ Then, we define the distance
\begin{equation}
\sigma_\beta(\mu_1,\mu_2):=\sup_{\varphi:\vertiii{\varphi}_\beta\leq 1} \int_{\mathbf X} \varphi(x)(\mu_1-\mu_2)(\mathrm{d}x).
\label{eq:sigma beta def}
\end{equation}
\end{definition}

It is readily seen that $\sigma_\beta(\mu_1,\mu_2)$ is a metric in the space of probability measures.
\noindent
A simple corollary of Proposition \ref{prop:sigma beta eq rho beta} is

\begin{corollary} 
	\label{cor:sigma beta eq rho beta for pairs} Let $\mu_1, \mu_2$ be two possibly signed measures on $\mathbf X$ as in Definition \ref{def:sigma beta for pairs}. Then
	\begin{equation}
	\sigma_\beta(\mu_1,\mu_2)\,=\,\rho_\beta(\mu_1,\mu_2).
\label{eq:sigma beta eq rho beta for pairs}
	\end{equation}
    \end{corollary}

Now we are in a position to summarize the main results of \cite{hairer2011yet}. It is well-known that the kernels $\P_\theta$ acting on probability measures are non-expansive in total variation distance:

\begin{equation}
| \P_\theta \mu_1 - \P_\theta \mu_2 |_{\rm TV} \le |  \mu_1 - \mu_2 |_{\rm TV}.   
\end{equation}
A major contribution of \cite{hairer2011yet} is the result stating that the kernels $\P_\theta$ are actually contractions in weighted total variation distance by choosing $\beta < \overline \alpha /K$, under simple conditions, see Proposition \ref{prop:Contraction of sigma beta} below. The path to proving this result is to establish first that the operator $\P_\theta^\ast$ acting on the Banach space ${\cal L}_{V,0}$ is a contraction, as stated in  
\cite[Theorem 3.1]{hairer2011yet}:

\begin{proposition}
\label{prop:Beta Norm Contraction by P}
Under Assumptions \ref{cond:Uniform Drift Condition} and \ref{cond:Local Minorization}, there is $\beta>0$ and $\alpha\in(0,1)$ such that for all $\theta$ and $\varphi \in {\cal L}_V$ 
\begin{equation}
\label{eq:Beta Norm Contraction by P}
\vertiii{\P^\ast_\theta \varphi}_\beta \,\leq\, \alpha\vertiii{\varphi}_\beta.
\end{equation}
The pairs $(\beta, \alpha)$ can be chosen as follows: take $\alpha_0 \in (0, \bar{\alpha})$ and $\gamma_0 \in (\gamma+2K / R, 1),$ and then set 
$$
\beta = \alpha_0/K \quad {\rm and } \quad \alpha = (1-(\bar{\alpha} - \alpha_0)) \vee (2+R \beta \gamma_0)/(2+R\beta).
$$
\end{proposition}

\begin{remark}
\label{rem:alpha Majorates gamma} 
Although there is a freedom in choosing $\alpha_0$ and $\gamma_0,$ the provable contraction coefficient $\alpha$, ensured by Proposition \ref{prop:Beta Norm Contraction by P}, can be easily shown to satisfy $\alpha > \gamma, $ i.e. not surprisingly, $\alpha$ is strictly larger than the contraction coefficient $\gamma$ in the drift condition, Assumption \ref{cond:Uniform Drift Condition}. 
\end{remark}

Proposition \ref{prop:Beta Norm Contraction by P} can be restated as saying that $P_\theta^\ast$ is a contraction on the Banach space ${\cal L}_{V,0}.$ But then its adjoint operator $P_\theta$, having the same norm, is also a contraction. Thus we get the what is essentially stated in \cite[Theorem 1.3]{hairer2011yet}: 

\begin{proposition}
\label{prop:Contraction of sigma beta}
Under the assumptions of Proposition \ref{prop:Beta Norm Contraction by P} there exist $\beta>0$ and $\alpha\in(0,1)$,  such that for all $\theta$, and any signed measure $\eta \in {\cal M}_{V}^0$ we have
\begin{equation}
\label{eq:Contraction of sigma beta with eta}
\sigma_\beta(\P_\theta \eta)\,\leq\, \alpha \sigma_\beta(\eta).
\end{equation}
Alternatively, let $\mu_1, \mu_2$ be two possibly signed measures on $\mathbf X$ as in Definition \ref{def:sigma beta for pairs}.
Then, we have
\begin{equation}
\label{eq:Contraction of sigma beta DEL}
\sigma_\beta(\P_\theta \mu_1,\P_\theta \mu_2)\,\leq\, \alpha \sigma_\beta( \mu_1, \mu_2).
\end{equation}
\end{proposition}

In what follows, $\beta$ and $\alpha$ are chosen as indicated in Proposition \ref{prop:Beta Norm Contraction by P}. Using standard arguments one can easily show the following proposition, also stated in \cite[Theorem 3.2]{hairer2011yet}:

\begin{proposition} 
	\label{prop:Invariant Measure}
	Under Assumptions \ref{cond:Uniform Drift Condition} and \ref{cond:Local Minorization} for all $\theta$ there is a unique  probability measure $\mu^\ast_\theta$ on $\mathbf X$ such that $\mu^\ast_\theta (V) = \int_{\mathbf X} V(x) \, \mu^\ast_\theta (\mathrm{d}x) <\infty$ and $\P_\theta \,\mu^\ast_\theta = \mu^\ast_\theta.$ 
\end{proposition}

\begin{remark}
\label{rem:HM vs MT}
A mirror image of Proposition \ref{prop:Beta Norm Contraction by P} given within \cite[Theorem 14.0.1]{meyn2012markov}, is geometric ergodicity, restated as\vspace{-2mm}
\begin{equation*}
\sup_{\| \varphi \|_\beta \le 1} \left( \mathbf E \, [\varphi (X_n) \, | X_0 = x] - \mathbf E_{\mu^\ast} \varphi (x)  \right) / (1 + \beta V(x)) \le C \alpha^n .  
\end{equation*}
It is readily seen that $\vertiii {\P^{\ast n}  \, \varphi}_\beta  \le C \alpha^n \| \varphi\|_\beta$ follows, but this is much weaker than Proposition \ref{prop:Beta Norm Contraction by P}. 
The conditions of \cite{meyn2012markov} are also much different by assuming that the Markov-chain is  $\psi$-irreducible and aperiodic. The drift condition is supplemented by a local minorization condition on a ``small set'' $\cal C$ defined in terms of an irreducibility measure $\psi$ so that $\psi({\cal C}) >0.$

\end{remark}

\section{Existence and Uniqueness of the Solution of a Poisson Equation}
\label{sec:Poisson Equation - Existence UniquenessMain results}

In what follows, we consider the Poisson equations, depending on the parameter $\theta \in \Theta$, 
\begin{equation}
\label{eq:Poisson theta}
(I-P_\theta^\ast)u_\theta(x) = f_\theta(x) - h_\theta,
\end{equation}
where 
$f_\theta :\mathbf X \to \mathbb{R}$ is the input data, $h_\theta=\mu^\ast_\theta(f_\theta),$ and $u_{\theta}:\mathbf X \to \mathbb{R}$ is the sought-after solution.

First, we prove the existence and the uniqueness (up to an additive constant) of the solution for a fixed $\theta$, adapting standard arguments, then we formulate smoothness conditions on the kernel $P_\theta^\ast$, and the right hand side, $f_\theta$. Using these conditions, we prove Lipschitz continuity w.r.t. $\theta$ in the norm $\|.\|_{\beta}$ of the particular solution $u_\theta$ for which $\mu^\ast_\theta(u_\theta)=0$. For a start, let $\theta \in \Theta$ be fixed.

\begin{theorem} 
\label{th:Poisson Eq Existence}
Let $\theta \in \Theta$ be fixed. Let $\P=\P_\theta$ be such that Assumptions \ref{cond:Uniform Drift Condition} and \ref{cond:Local Minorization} hold. Let $\beta > 0$ be as given in Proposition \ref{prop:Beta Norm Contraction by P}, and  
let $\mu^\ast$ denote the unique invariant probability measure of $\P.$ Let $f: {\mathbf X} \rightarrow \mathbb{R}$ be a measurable function 
such that $ \vertiii{f}_\beta<\infty$, and let $h=\mu^\ast(f).$ Then, the Poisson equation
\begin{equation}
\label{eq:Poisson eq def}
(I-\Ps)u(x) = f(x) - h
\end{equation}
has a unique solution $u(\cdot)$ up to an additive constant. The particular solution for which $\mu^\ast(u) = 0$ can be written as
\begin{equation}
u(x) = \sum_{n=0}^{\infty} (P^{\ast n}  f(x) - h),
\label{eq:Poisson Solution u as Inf Sum}
\end{equation}
where the right hand side is absolutely convergent, and
\begin{equation}
|u(x)|\,\leq\,  \vertiii{f}_\beta \hspace{0.4mm} K_u (1 + \beta V(x)), 
\label{eq:Poisson eq ux upper bound}
\end{equation}
for some constant $K_{u}>0$ depending only on the constants appearing in Assumptions \ref{cond:Uniform Drift Condition}, \ref{cond:Local Minorization}, given by 
\begin{equation*}
K_u := \frac{1}{1-\alpha} \left( 2 + \beta \frac{K}{1-\gamma} \right).
\end{equation*}
It also follows that $\|u\|_\beta \,\leq\,  K_{u} \vertiii{f}_\beta < \infty.$
\end{theorem}

\begin{proof}
It is immediate to check that \eqref{eq:Poisson eq def} is formally satisfied by $u$. 
We show that $u$ is well-defined. First, 
consider any function $\varphi$ such that $\vertiii{\varphi}_\beta\leq 1$. By the definition of the metric $\sigma_\beta$, see \eqref{eq:sigma beta def}, the inequality 
\begin{equation}
\left|\int_{\mathbf X} \varphi(x)(\mu_1-\mu_2)(\mathrm{d}x) \right| \leq \sigma_\beta(\mu_1,\mu_2) 
\end{equation}
holds true for any pair of probability measures $\mu_1,\mu_2,$ or even for any pair of signed measures $\mu_1, \mu_2$ as in Definition \ref{def:sigma beta for pairs}. On the other hand, any generic function $\varphi$ can be rescaled by $\frac{1}{\vertiii{\varphi}_\beta}$, so that we also have
\begin{equation}
\left| \int_{\mathbf X} \varphi(x)(\mu_1-\mu_2)(\mathrm{d}x) \right| \leq  \vertiii{\varphi}_\beta  \sigma_\beta(\mu_1,\mu_2).
\label{eq:Heart Ineq with Triple Norm}
\end{equation} 

To estimate 
the $n$\,th term of the right hand side of \eqref{eq:Poisson Solution u as Inf Sum},  
consider the equalities
\begin{eqnarray}
\frac{1}{\vertiii{f}_\beta} |P^{\ast n} f(x)-\mu^\ast(f)| = 
\frac{1}{\vertiii{f}_\beta} | (\P^n \delta_x- \mu^\ast)(f)|\nonumber\\
= \frac{1}{\vertiii{f}_\beta} \left| \int_{\mathbf X} f(y)(\P^n \delta_x-\P^n \mu^\ast)(\mathrm{d}y)\right|.
\label{eq:P*n f- mu* f}
\end{eqnarray}

Using 
\eqref{eq:Heart Ineq with Triple Norm}, we can bound the right hand side by
$\sigma_\beta(\P^n\delta_x,\P^n\mu^\ast)$.
Now applying Proposition \ref{prop:Contraction of sigma beta} and taking into account Corollary 
\ref{cor:sigma beta eq rho beta for pairs}, we can further bound it by
\begin{align}
\!\!\!\sigma_\beta(\P^n\delta_x,\P^n\mu^\ast) &\leq  \alpha^n\sigma_\beta(\delta_x,\mu^\ast)\nonumber\\
& =  \alpha^n\!\!\!\!\! \sup_{\varphi:\|\varphi\|_\beta\leq 1} \int_{\mathbf X} \varphi(x) (\delta_x-\mu^\ast)(\mathrm{d}x).
\end{align}
Take into account the trivial estimate
\begin{equation}
\int_{\mathbf X} \varphi(x) (\delta_x-\mu^\ast)(\mathrm{d}x)\leq  \int_{\mathbf X} |\varphi(x)| (\delta_x+\mu^\ast)(\mathrm{d}x),
\label{eq:varphi (delta x - mu *}
\end{equation}
and note that $\|\varphi\|_\beta\leq 1$ implies 
$|\varphi(x)|\leq 1+\beta V(x)$ for all $x$ .  
Putting together \eqref{eq:P*n f- mu* f} - \eqref{eq:varphi (delta x - mu *} with the fact that 
$\int_{\mathbf X} (1+\beta V(x)) (\delta_x+\mu^\ast)(\mathrm{d}x)=2+\beta V(x)+\beta \mu^\ast(V)$, we conclude:
\begin{equation}
\frac{1}{\vertiii{f}_\beta} |P^{\ast n} f(x)-\mu^\ast(f)| \leq  \alpha^n (2+\beta V(x) + \beta \mu^\ast(V)).
\label{eq:bound1ofPf-m}
\end{equation}

It follows that the series $\sum_{n=0}^{\infty} (P^{\ast n}  f(x) - h)$ is absolutely convergent, so $u(x)$ is well-defined
and satisfies the desired upper bound. 
Indeed, $(\Ps u)(x)$ can be written as
\begin{equation}
\int_{\mathbf X} P(x, \dd y) u(y) = \int_{\mathbf X} P(x, \dd y) \sum_{n=0}^{\infty} (P^{\ast n}  f(y) - h), 
\end{equation}
where the integration and the summation can be interchanged due to the Lebesgue dominated convergence theorem, the conditions of which are ensured by \eqref{eq:bound1ofPf-m}. Thus, we get 
\vspace{-1mm}
\begin{equation}
(\Ps u)(x) = \sum_{n=1}^{\infty} (P^{\ast n}  f(x) - h) = u(x) - (f(x) - h), 
\end{equation}
which implies the claim. Using similar arguments, and Fubini's theorem as in \eqref{eq:Fubini for mu P phi} of the Appendix, we get that
\begin{equation}
\int_{\mathbf X} u(x) \mu^\ast (\mathrm{d}x) = 0.
\end{equation}
To prove uniqueness, assume that there are two solutions $u_1$ and $u_2$, and define $\Delta u=u_2-u_1$. Then, $(I-\Ps)\Delta u=0$, implying $\Ps\Delta u= \Delta u$, from which $\vertiii{\Ps\Delta u}_\beta=\vertiii{\Delta u}_\beta$. But, by Proposition \ref{prop:Beta Norm Contraction by P}, it holds that $\vertiii{\Ps\Delta u}_\beta\leq \alpha \vertiii{\Delta u}_\beta,$ and hence    $\vertiii{\Delta u}_\beta=0$. Therefore, $\Delta u$ is a constant.

Summing the inequalities  \eqref{eq:bound1ofPf-m} over $n$ and using \eqref{eq:VdmuBound} we get
\begin{equation}
|u(x)|\leq \frac{\vertiii{f}_\beta}{1-\alpha} \left(2 +\beta V(x) +  \beta \frac{K}{1-\gamma}\right),
\end{equation}
from which the claim of the theorem follows after trivial arithmetics.
\end{proof}

\section{Lipschitz Continuity of the Kernel}
\label{sec:Lipschitz_Continuity_of_the_Kernel}

Now we consider a parametric family of kernels $(\P_{\theta})$.
A critical point in the discussion to follow is to define appropriate smoothness conditions for them in the context of \cite{hairer2011yet}.

\smallskip
\begin{assumption} [Lipschitz Continuity of $P_\theta$]
\label{cond:Lipschitz of P theta Dirac x by Miklos}
There exists a constant $L_P$ such that for every $\theta,\theta^{\prime}\in \Theta$ and all $x \in \mathbf X:$ 
\begin{equation}
\sigma_\beta(\P_{\theta}\delta_x,\P_{\theta^{\prime}}\delta_x)\,\leq\, L_P|\theta-\theta^{\prime}|(1+\beta V(x)).
 \label{eq:Miklos Lipschitz Dirac x}
 \end{equation} 
\end{assumption}
This assumption can be rewritten in the equivalent form: for any $f \in \cal L_V$ we have, with $\beta >0$ as in Proposition 2, 
\begin{equation*}
\label{eq:Miklos Lipschitz Dirac x with P star}
\vert \P_\theta \delta_x \, (f)- \P_{\theta^\prime} \delta_x (f) \vert \le L_P \| f \|_\beta \vert \theta - \theta^\prime \vert (1 + \beta V(x)).
\end{equation*}

An inequality analogous to \eqref{eq:Miklos Lipschitz Dirac x} with a  general measure $\mu$ replacing $\delta_x$ is established in the following lemma:

\begin{lemma}
\label{lemma:P theta Lipschitz with mu}
Let Assumption \ref{cond:Lipschitz of P theta Dirac x by Miklos} be satisfied, and suppose that Assumption \ref{cond:Uniform Drift Condition} holds, without requiring $\gamma < 1.$ Let $\mu$ be a measure such that $\mu(1+\beta V) < \infty.$ Then for every $\theta,\theta^{\prime}\in\Theta:$
\begin{equation}
\sigma_\beta(\P_{\theta}\mu,\P_{\theta^{\prime}}\mu)\,\leq\, L_P|\theta-\theta^{\prime}|\mu(1+\beta V).
\label{eq:P theta Lipschitz with mu}
\end{equation}
\end{lemma}

The starting point of the proof is the observation that 
Assumption \ref{cond:Lipschitz of P theta Dirac x by Miklos} implies that for all $\varphi$ such that $\|\varphi\|_{\beta} \leq 1$, implying also $\vertiii{\varphi}_{\beta} \leq 1$, we have 
\begin{equation}
\int_{\mathbf X} \varphi(y) \left(\P_{\theta}(x,\dd y) - P_{\theta^{\prime}}(x,\dd y)\right) \leq L_P|\theta-\theta^{\prime}|(1+\beta V(x)).
\label{eq:Miklos Lipschitz Dirac x Applied to varphi}
\end{equation}
Integrating this inequality with respect to $\mu(\mathrm{d}x)$ the right hand side of \eqref{eq:Miklos Lipschitz Dirac x Applied to varphi} becomes the right hand side of \eqref{eq:P theta Lipschitz with mu}. For the integral of the left hand side we apply Fubini's theorem to get the claim of the lemma. Details will be given in the Appendix.

The relaxed version of Assumption \ref{cond:Uniform Drift Condition} not requiring $\gamma < 1$ will be referred to as {\it uniform one step growth condition}. It is analogous to {Assumption A'\!.5, (i') on page 290} in \cite{benveniste1990}, and will play a dominant role in Section \ref{sec:Relaxations of the uniform drift condition} below. It is readily seen that the corollaries of Assumption \ref{cond:Uniform Drift Condition} given as inequalities \eqref{eq:Uniform Drift Condition w mu} and \eqref{eq:Uniform Drift Condition w eta} remain valid.

Lemma \ref{lemma:P theta Lipschitz with mu} readily extends to signed measures:

\begin{lemma}
\label{lemma:P theta Lipschitz wrt signed eta}
Let the conditions of Lemma \ref{lemma:P theta Lipschitz with mu} be satisfied and let 
$\eta$ be a signed measure such that  $|\eta|(1+\beta V) < \infty.$ Then:
\begin{equation}
\sigma_\beta(\P_{\theta}\eta,\P_{\theta^{\prime}}\eta)\,\leq\, L_P|\theta-\theta^{\prime}| |\eta|(1+\beta V).
\label{eq:MiklosLipscitzSigned}
\end{equation}
\end{lemma}

The proof is based on using the Hahn-Jordan decomposition $\eta = \eta^+ - \eta^-$, where $\eta^+$ and $\eta^-$ are non-negative measures with disjoint supports. Details will be given in the Appendix. 
The previous results culminate in what follows, stating the Lipschitz continuity of $\P^n_{\theta}$ when acting on signed measures: 
\begin{lemma}
\label{lemma:Lipschitz for Pn with eta}
Let $(\P_\theta)$ satisfy Assumptions  \ref{cond:Uniform Drift Condition} and 
\ref{cond:Local Minorization}. Let $\beta >0$ be such that  Proposition \ref{prop:Contraction of sigma beta} holds. 
Let Assumption \ref{cond:Lipschitz of P theta Dirac x by Miklos}, requiring the Lipschitz continuity of $(\P_\theta),$ hold with the above $\beta.$ Then for any signed measure $\eta$ with $|\eta|(1+\beta V) < \infty$ and $\theta,\theta^{\prime}\in\Theta$:
\begin{equation}
\label{eq:Lipschitz for Pn with any mu}
\sigma_\beta(\P^n_{\theta}\eta,\P^n_{\theta^{\prime}}\eta)\, \leq\, L_P |\theta-\theta^{\prime}|  \cdot
C_P |\eta| ( 1 + \beta V),
\end{equation}
where $C_P$ is independent of $\theta$, $\theta'$ and $\eta$, and is given by
\begin{equation}
\label{eq:LPprime}
C_P := \frac{1}{\-1-\alpha}\left(1 + \beta \frac{K}{1-\gamma}\right) \vee \frac{1}{\alpha-\gamma}.
\end{equation}
\end{lemma}

The proof is based on the telescopic sum decomposition: 
\begin{align}
\label{eq:P1n vs P2n Telecopic}
(\P^n_{\theta} - \P^n_{\theta^{\prime}}) = \sum_{k=0}^{n-1}( \P^{n-k}_{\theta}\P^{k}_{\theta^{\prime}} - \P^{n-k-1}_{\theta}\P^{k+1}_{\theta^{\prime}}  ).
\end{align}

For the $k$-th term we use the contraction property of 
$\P^{n-k-1}_{\theta}$, Proposition \ref{prop:Contraction of sigma beta}, and estimate $\sigma_\beta( \P_{\theta}\P^{k}_{\theta^{\prime}} \eta - \P_{\theta^{\prime}}\P^{k}_{\theta^{\prime}}\eta  )$ from above using the Lipschitz-continuity of the kernels $\P_{\theta}$ as defined in Assumption \ref{cond:Lipschitz of P theta Dirac x by Miklos}. We will also need the observation that iterating the drift condition given in the form \eqref{eq:Uniform Drift Condition w eta} we get  
\begin{align}
|\P_\theta^k\eta|(V) \le \gamma^k |\eta|(V)+\frac{K}{1-\gamma} |\eta|(\mathbf X),
\label{eq:Pk mu V upper bound V0}
\end{align}
restated in \eqref{eq:Pk mu V upper bound} with details given in the Appendix.

Applying Lemma \ref{lemma:Lipschitz for Pn with eta} for $\eta = \delta_x$, we get:
\begin{equation}
\sigma_\beta(\P^n_{\theta}\delta_x,\P^n_{\theta^{\prime}}\delta_x) \leq L_P |\theta-\theta^{\prime}|  \cdot
C_P  ( 1 + \beta V(x)).
\end{equation}
Letting $n \rightarrow \infty,$ and taking into account Proposition \ref{prop:Contraction of sigma beta}, we get 
the Lipschitz continuity of the invariant measure 
w.r.t.\ $\theta$:

\begin{corollary}
\label{corr:Lipschitz for mu stars}
Under the assumptions of Lemma \ref{lemma:Lipschitz for Pn with eta}, we get
\begin{equation}
\label{eq:Lipschitz for mu stars}	
\sigma_\beta(\mu_{\theta}^\ast,\mu_{\theta^{\prime}}^\ast) \,\leq\, L_P |\theta-\theta^{\prime}| {\cdot C_P}.
\end{equation}
The constant $C_P$ can be replaced by the smaller constant 
\begin{equation}
C'_P:=\frac{1}{\-1-\alpha}\left(1 + \beta \frac{K}{1-\gamma}\right).
\end{equation}
\end{corollary}
\noindent
Details of the proof will be given in the Appendix. 

A surprising variant of the above lemma is the following:

\begin{lemma}
\label{lemma:Lipschitz for Pn etaX 0} 
Under the conditions of Lemma \ref{lemma:Lipschitz for Pn with eta} for every signed measure $\eta \in \mathcal{M}_V^0$, implying 
$\eta(\mathbf X) = 0$, and $\theta,\theta^{\prime}\in\Theta$, 
we have 
	\begin{equation}
	\label{eq:MiklosLipschitzPn etaX 0}
	\sigma_\beta(\P^n_{\theta}\eta,\P^n_{\theta^{\prime}}\eta)\, \leq \,
	L_P|\theta-\theta^{\prime}| n \alpha^{n-1}\, |\eta|(1+ \beta V).
	\end{equation}
\end{lemma}
Equivalently, we can write: for any $f \in \cal L_V,$ we have
\begin{equation*}
\label{eq:MiklosLipschitzPn etaX 0}
\vert P^n_{\theta}\eta \, (f) - \P^n_{\theta^{\prime}}\eta\, (f) \vert \leq \,
L_P \| f \|_\beta \,|\theta-\theta^{\prime}| n \alpha^{n-1}\, |\eta|(1+ \beta V).
\end{equation*}

The starting point of the proof is the same telescopic decomposition \eqref{eq:P1n vs P2n Telecopic}, applied to $\eta.$ The novelty w.r.t. to the proof of Lemma \ref{lemma:Lipschitz for Pn with eta} is the observation that $\eta(\mathbf X) = 0$ implies that $\P^{k}_{\theta^{\prime}} \eta$ converges exponentially fast to the zero measure, see Proposition \ref{prop:Contraction of sigma beta}. 
Details will be given in the Appendix.

A final note: a key technical result of \cite{benveniste1990} on the Lipschitz-continuity of the solutions of the Poisson equation, stated as Theorem 6, p.\, 262, assumes smoothness of the kernels in a way, which is in most aspects significantly more restrictive than our Assumption \ref{cond:Lipschitz of P theta Dirac x by Miklos}. 
In particular, highlighting the main features of their key assumptions (iii) and (iv), modulated to the context and notations of our paper, would be as follows.
Let $\mathbf X : = \mathbb {R}^k$ be a Euclidean space, and let $V(x) = |x|^{p}$ with $p >0.$ Letting $\beta=1$ write $\|{g} \|_{1} =:\|{g} \|_{V}.$ Let $g(\cdot)$ be differentiable, and let ${g}^\prime(\cdot)$ denote its gradient. Then the modulated version of assumption  (iii) would read: for any pair $\theta,\, \theta^\prime \in \Theta$ and any $x \in \mathbb {R}^k$ we have 
\begin{equation*}
	\label{eq:Assump iii}
	\vert P^n_{\theta}\delta_x \, (g) - \P^n_{\theta^{\prime}}\delta_x\, (g) \vert \leq \,
	L_P (\|{g} \|_{V}+ \|{g}^\prime \|_{V}) \,|\theta-\theta^{\prime}| \, (1+ V(x)).
	\end{equation*}
To modulate (and relax) assumption (iv) let $\eta \in {\cal M}_V^0$ and let the test functions $g(\cdot)$ be twice  differentiable. Then impose: 
\begin{equation*}
	\left| \P^{n}_{\theta} \eta \, (g) - \P^{n}_{\theta'} \eta \, (g) \right| \le L_p\, \rho^n \, \|{g} \|_{3, V} \,  |\, \theta - \theta'| \, | \eta | (1 +  V ),
  \end{equation*}
  with $0 <\rho < 1$ and $\|{g} \|_{3, V}  = (\|{g} \|_{V}+ \|{g}^\prime \|_{V} + \|{g}^{\prime \prime} \|_{V}) .$ Note in passing that assumption (iv) in \cite{benveniste1990} is formulated in terms of measures of the form $\eta = \delta_{x_1} - \delta_{x_2},$ however it can be readily extended to measures $\eta \in {\cal M}_V^0$ via Choquet's theorem. 
  
  The strength of these assumptions, or those of \cite{benveniste1990}, is that the set of test functions are strictly smaller than ${\cal L}_V$, needed for our Assumption \ref{cond:Lipschitz of P theta Dirac x by Miklos}.  On the other hand, in \cite{benveniste1990} the distances of measures $\P_\theta^n \eta$ and $\P_{\theta^\prime}^n \eta$, in any suitable metric, are not estimated. More grievously, both assumptions of \cite{benveniste1990} assume some kind of a priori stability of the kernels $\P^{n}_{\theta},$ while similar inequalities are actually proven in our Lemmas \ref{lemma:Lipschitz for Pn with eta} and \ref{lemma:Lipschitz for Pn etaX 0} under practically attractive assumptions.

\section{Lipschitz Continuity of the  Solution of the Poisson Equation}
\label{sec:Poisson Equation - Lipschitz Continuity}

In this section we revisit the family of Poisson equations, depending on the parameter $\theta \in \Theta$, defined under \eqref{eq:Poisson theta}. In addition to the Lipschitz-continuity of the kernel $P_\theta^\ast$, given in Assumption \ref{cond:Lipschitz of P theta Dirac x by Miklos}, we need to 
formulate the Lipschitz-continuity of the right hand side, $(f_{\theta})$ with $\theta \in \Theta$, as well.

\begin{assumption}
\label{cond:f theta beta Lipschitz} [Lipschitz Continuity of $f_\theta$]
We have
$K_f := \sup_{\theta\in\Theta}\vertiii{f_\theta}_\beta<\infty$, and
there exists a constant $L_f$ such that, for all $\theta,\theta'$, it holds that 
\begin{equation}
\|f_\theta - f_{\theta'}\|_\beta \, \leq\, L_f |\theta-\theta'|.
\label{eq:f theta beta Lipschitz}
\end{equation}
\end{assumption}

The main result of the present paper is that under the above conditions the particular solution $u_\theta$, for which $\mu^\ast_\theta(u_\theta)=0$, is Lipschitz continuous w.r.t. $\theta$ in the norm $\|.\|_{\beta}$:

\begin{theorem} 
\label{thm:Poisson Lipschitz}
Assume that the kernels $(\P_\theta)$ satisfy Assumptions \ref{cond:Uniform Drift Condition}, \ref{cond:Local Minorization} and \ref{cond:Lipschitz of P theta Dirac x by Miklos}.  Let us fix $\beta > 0$ as given in Proposition \ref{prop:Beta Norm Contraction by P}. Let $(f_\theta)$ be a family of  ${\mathbf X} \rightarrow \mathbb{R}$ measurable functions such that Assumption \ref{cond:f theta beta Lipschitz} holds with the above $\beta$. Let $\mu_\theta^\ast$ denote the unique invariant probability measure of $\P_\theta$ and let $h_\theta=\mu_\theta^\ast(f_\theta).$ Consider the  Poisson equations
\begin{equation}
\label{eq:Poisson eq with theta}
(I-P_\theta^\ast)u_\theta(x) = f_\theta(x) - h_\theta.
\end{equation}
Then, $h_\theta$ is Lipschitz continuous in $\theta$:
\begin{equation}
|h_\theta-h_{\theta'}|\,\leq\, L_h |\theta-\theta'|, \\
\label{eq:Poisson h theta psi Lipschitz}
\end{equation}
and the particular solution, given in Theorem \ref{th:Poisson Eq Existence} by \eqref{eq:Poisson Solution u as Inf Sum} as $
				u_\theta(x)=\sum_{n=0}^{\infty} (P_\theta^{\ast n}  f_\theta(x) - h_\theta)$ 					
				is Lipschitz continuous in $\theta$:
\begin{equation}
|u_\theta(x)-u_{\theta'}(x)|\,\leq\, L_u |\theta - \theta^{\prime}| \left( 1  + \beta V(x)   \right) ,
\end{equation}
where $L_{u}$ is independent of $x$. Alternatively, we can write 
\begin{equation}
\|u_\theta - u_{\theta'}\|_\beta \,\le\, L_u |\theta-\theta'|.
\end{equation}
Here the constants $L_{h}$ and $L_{u}$ depend only on the constants appearing in Assumptions  \ref{cond:Uniform Drift Condition}, \ref{cond:Local Minorization},  \ref{cond:Lipschitz of P theta Dirac x by Miklos} and \ref{cond:f theta beta Lipschitz}.
\end{theorem}
\begin{proof}
Consider the extended parametric family of Poisson equations, where
$\P^{\ast}$ 
and $f$ are independently parametrized, with the notation $h_{\theta, \psi} = \mu_\theta^{\ast}(f_{\psi}),$
\begin{equation}
\label{eq:Poisson with theta psi}
(I-P_\theta^\ast)u_{\theta, \psi}(x) = f_{\psi}(x) - h_{\theta, \psi},
\end{equation}
{\it Step 1.} First, we prove that $h_{\theta, \psi}$ is Lipschitz continuous in $\theta$ and $\psi.$ Since $h_{\theta} = \mu_\theta^{\ast}(f_{\theta}) = h_{\theta, \theta}$, the Lipschitz continuity of $h_{\theta},$ stated in \eqref{eq:Poisson h theta psi Lipschitz} then follows. We can write
\begin{align} 
\label{eq:Poisson with theta psi h Lipschitz in psi}
|h_{\theta, \psi}-h_{\theta, \psi^{\prime}}| & =&\!\!\!\!  \lim_{n\rightarrow \infty} |{\P}_{\theta}^{\ast n} f_{\psi}(x)-{\P}_{\theta}^{\ast n} f_{\psi^{\prime}}(x)|, \\
\label{eq:Poisson with theta psi h Lipschitz in theta}
|h_{\theta, \psi}-h_{\theta^{\prime}, \psi}| & =&\!\!\!\! \lim_{n\rightarrow \infty} |{\P}_{\theta}^{\ast n} f_{\psi}(x)-{\P}_{\theta^{\prime}}^{\ast n} f_{\psi}(x)|. 
\end{align}
Note that the limits of the right hand side are finite by Assumption \ref{cond:f theta beta Lipschitz} and 
the drift condition Assumption \ref{cond:Uniform Drift Condition}.

We can bound the right hand side of \eqref{eq:Poisson with theta psi h Lipschitz in psi} as follows:
\begin{equation*}
|{\P}_{\theta}^{\ast n} f_{\psi}(x)-{\P}_{\theta}^{\ast n} f_{\psi^{\prime}}(x)| \leq \left({\P}_{\theta}^{\ast n}  |f_{\psi}-f_{\psi^{\prime}}|\right)(x)
\end{equation*}
\begin{equation}
 =  \left({\P}_{\theta}^{n}\delta_x\right) |f_{\psi}-f_{\psi^{\prime}}|
\leq  \left\| f_{\psi}-f_{\psi^{\prime}}\right\|_{\beta} \left({\P}_{\theta}^{n}\delta_x\right)(1+\beta V).
\end{equation}
Using the Lipschitz continuity of $f$, as given by Assumption \ref{cond:f theta beta Lipschitz},
the right hand side can bounded from above by 
	\begin{equation}
    \label{eq:Lips_const_wrt_psi}
		L_f|\psi-\psi^{\prime} | \left({\P}_{\theta}^{n}\delta_x\right) (1+\beta V).
	\end{equation}

Let $n \to \infty$, note that $\vertiii{{\P}_{\theta}^{n}\delta_x-\mu^\ast_{\theta}}_{\beta} \to 0$ by Corollary \ref{cor:sigma beta eq rho beta for pairs}, and thus 
$\|{\P}_{\theta}^{n}\delta_x-\mu^\ast_{\theta}\|_{\beta} \to 0,$  we get for the limit of \eqref{eq:Lips_const_wrt_psi}, yielding an upper bound for the l.h.s. (left hand side) of \eqref{eq:Poisson with theta psi h Lipschitz in psi}, by Remark \ref{rem:int V wrt mu star},
\begin{align}	
		L_f|\psi-\psi^{\prime} | \mu^\ast_{\theta}(1+\beta V)
\leq L_f|\psi-\psi^{\prime}| \left(1+\beta\frac{K}{1-\gamma}\right)\!.
\label{eq:Poisson with theta psi h Lipschitz in psi Details3}
	\end{align}

Now, the l.h.s. of \eqref{eq:Poisson with theta psi h Lipschitz in theta} can be written as and bounded by
\begin{equation}
\label{eq:Poisson with theta psi h Lipschitz in theta Details2}
\left| \int_{\mathbf X} f_{\psi}(x) ({\mu^\ast_\theta} - {\mu^\ast_{\theta'}})(\mathrm{d}x) \right| \leq \vertiii{f_{\psi}}_{\beta} \sigma_\beta(\mu_{\theta}^\ast,\mu_{\theta^{\prime}}^\ast).
\end{equation}
Here $\vertiii{f_{\psi}}_\beta\leq\sup_{\psi\in \Theta}\vertiii{f_{\psi}}_\beta = K_f <\infty$ by Assumption \ref{cond:f theta beta Lipschitz}, and $\sigma_\beta(\mu_{\theta}^\ast,\mu_{\theta^{\prime}}^\ast) \le L_P |\theta-\theta^{\prime}| \cdot C_P $ by Corollary \ref{corr:Lipschitz for mu stars}.

Setting $\psi = \theta, \psi^\prime= \theta^\prime$ in \eqref{eq:Poisson with theta psi h Lipschitz in psi} and $\psi = \theta^\prime$ in \eqref{eq:Poisson with theta psi h Lipschitz in theta}, we get by the triangle inequality \eqref{eq:Poisson h theta psi Lipschitz}:
$|h_{\theta}-h_{\theta^{\prime}}| \leq L_h |\theta - \theta^{\prime}|$, with
\begin{align}
L_h &= L_f \left(1+\beta\frac{K}{1-\gamma}\right) + K_f L_P \frac{1}{\-1-\alpha}\left(1 + \beta \frac{K}{1-\gamma}\right)\nonumber\\
& = \left(L_f + K_f L_P \frac{1}{\-1-\alpha} \right)\left(1 + \beta \frac{K}{1-\gamma}\right).  \label{eq:Poisson Lipschitz constant for h}
\end{align}

\ms
With this Step 1 is completed. 
Next, we consider the Lipschitz continuity of the doubly-parametrized particular solution 
\begin{equation}
u_{\theta, \psi}(x)=\sum_{n=0}^{\infty} (P_\theta ^{\ast n}  f_\psi(x) - h_{\theta, \psi}).
\end{equation}

	{\it Step 2}. We show that $u_{\theta, \psi}(x)$ is Lipschitz continuous w.r.t. $\psi.$ Indeed, we can express $|u_{\theta, \psi}(x) - u_{\theta, \psi^{\prime}} (x)|$ as
	\begin{align}
    \label{eq:Poisson Lipschitz of u wrt 	psi D1}
&	\Big|\sum_{n=0}^{\infty} (P_\theta^{\ast n}  (f_{\psi}(x) - f_{\psi^{\prime}}(x)) - (h_{\theta, \psi} - h_{\theta, \psi^{\prime}})\Big|\nonumber\\ 
=\;& \Big|\sum_{n=0}^{\infty} (\P_\theta^{n} \delta_x  - \mu_\theta^\ast)  (f_{\psi} - f_{\psi^{\prime}}) \Big|.
	\end{align}
	The absolute value of $n$-th term can be written, using 
	\eqref{eq:Heart Ineq with Triple Norm}, as 
	\begin{equation*}
	|(\P_\theta^{n} \delta_x  - \P_\theta^{n} \mu_\theta^\ast) (f_{\psi} - f_{\psi^{\prime}})  | \le  \sigma_\beta(\P_\theta^{n} (\delta_x  - \mu_\theta^\ast))~\vertiii{(f_{\psi} - f_{\psi^{\prime}})}_\beta.
	\end{equation*} 
    \noindent Taking into account Propositions \ref{prop:Contraction of sigma beta} and \ref{prop:sigma beta eq rho beta}, \eqref{def:rhobetaeta},
    and Assumption \ref{cond:f theta beta Lipschitz}, the right hand side can be bounded from above by 
    \begin{align}
	&  \alpha^n \sigma_\beta(\delta_x  - \mu_\theta^\ast) \cdot L_f |\psi - \psi^{\prime}|\nonumber\\
	& \le  \alpha^n (\delta_x  +  \mu_\theta^\ast)(1 + \beta V) \cdot L_f |\psi - \psi^{\prime}|.
    \end{align}
    The right hand side is equal to and can be upper bounded by
    \begin{align}
	& \alpha^n \left(2 + \beta V(x)  +  \beta \mu_\theta^\ast (V) \right)\cdot L_f |\psi - \psi^{\prime}|\nonumber \\
	& \le \alpha^n \left(2 + \beta V(x)  +  \beta \frac{K}{1-\gamma} \right)\cdot L_f |\psi - \psi^{\prime}|.
	\end{align}
    Inserting this into \eqref{eq:Poisson Lipschitz of u wrt psi D1} we get the upper bound
    \begin{align}
	  & |u_{\theta, \psi}(x) - u_{\theta, \psi^{\prime}} (x)| \nonumber \\
    & \leq {\frac 1 {1 - \alpha}} \left(2 + \beta V(x) + \beta \frac{K}{1-\gamma}  \right) L_f |\psi - \psi^{\prime}|.
    \label{eq:Two parameter Poisson Lipschitz of u wrt psi}
    \end{align}

{\it Step 3}. The final, critical point is to show that $u_{\theta, \psi}(x)$ is Lipschitz continuous in $\theta.$  Let us write $u_{\theta, \psi}(x)-u_{\theta^{\prime},\psi}(x)$ as
	\begin{equation}
	\label{eq:Two parameter Poisson Lipschitz of u wrt theta D1}	
	\sum_{n=0}^{\infty} (\P_{\theta}^{\ast n} f_\psi(x)- \mu_\theta^\ast (f_\psi) -  \P_{\theta^{\prime}}^{\ast n} f_\psi(x) +  \mu_{\theta^{\prime}}^\ast (f_\psi) ).
	\end{equation}
	The $n$-th term can be written as 
	\begin{equation}
	\label{eq:f psi Delta n}
	\left( \P_{\theta}^{ n} \delta_x - \P_{\theta}^{ n} \mu_\theta^\ast  -  \P_{\theta^{\prime}}^{ n} \delta_x  +  \P_{\theta^{\prime}}^{ n} \mu_{\theta^\prime}^\ast \right)(f_\psi) .
	\end{equation}
    Let us denote the measure acting on $f_\psi$ by $\Delta_n.$ 
	Adding and subtracting $\P_{\theta^{\prime}}^{ n} \mu_{\theta}^\ast$ within $\Delta_n$, we can write 
    \begin{align}
  	\Delta_n & = [\P_{\theta}^{ n} \left( \delta_x - \mu_\theta^\ast \right)  -  \P_{\theta^{\prime}}^{ n} \left(\delta_x  -  \mu_\theta^\ast \right) ] + 
	[\P_{\theta^{\prime}}^{ n} {\color{black} \left(\mu_{\theta^\prime}^\ast -  \mu_{\theta}^\ast \right)} ] \nonumber \\
    & =: \Delta_{n,1} + \Delta_{n,2} .
    \label{eq:splitting-delta-n}
    \end{align}
With this notation 
\eqref{eq:f psi Delta n}, the $n$-th term of \eqref{eq:Two parameter Poisson Lipschitz of u wrt theta D1}, can be written and bounded from above 
in absolute value, using 
\eqref{eq:Heart Ineq with Triple Norm},
as 
	\begin{equation}
    \label{eq:Delta n acting on f psi le 01}
	| {\Delta_{n}( f_{\psi})}| \le \sigma_\beta (\Delta_{n}) \vertiii{f_{\psi} }_\beta = \sigma_\beta (\Delta_{n,1}+\Delta_{n,2}) \vertiii{f_{\psi} }_\beta.     
	\end{equation}
To bound $\sigma_\beta  (\Delta_{n,1}) $ we can apply Lemma \ref{lemma:Lipschitz for Pn etaX 0} with $\eta = \delta_x - \mu_\theta^\ast :$ 
    \begin{equation}    
    \sigma_\beta (\Delta_{n,1})  \le L_P |\theta-\theta^\prime|n {\alpha^{n-1}} 	|\delta_x - \mu_\theta^\ast|(1 + \beta V).
    \end{equation}
    The right hand side can be trivially upper bounded by
	\begin{align}
    \label{eq:rho beta Delta 1n}
    & L_P |\theta-\theta^\prime| n {\alpha^{n-1}} \left(2 + \beta V(x)  +  \beta \mu_\theta^\ast (V)    \right) \nonumber\\
	\le\; & L_P |\theta-\theta^\prime| n {\alpha^{n-1}} \left(2 + \beta V(x)  +  \beta \frac{K}{1-\gamma} \right).
	\end{align}

To bound $\sigma_\beta  (\Delta_{n,2})$, we refer to  Proposition \ref{prop:Contraction of sigma beta}, yielding\ 
$$\sigma_\beta (\Delta_{n,2}) \le  {\alpha^{n}} \sigma_\beta (\mu_{\theta^\prime}^\ast -  \mu_{\theta}^\ast).$$ 
This can be bounded from above by Corollary \ref{corr:Lipschitz for mu stars}, resulting in  
	\begin{equation}
	\label{eq:sigma beta Delta 2n le}
	\sigma_\beta (\Delta_{n,2}) \le  {\alpha^{n}} L_P ~ \frac{1}{\-1-\alpha}\left(1 + \beta \frac{K}{1-\gamma}\right) |\theta-\theta^\prime|.
	\end{equation}

    Combining \eqref{eq:Delta n acting on f psi le 01} - \eqref{eq:sigma beta Delta 2n le} we get that the $n$-th term of \eqref{eq:Two parameter Poisson Lipschitz of u wrt theta D1}, reformulated as \eqref{eq:f psi Delta n}, can be bounded from above by  
	\begin{align}
	&  L_P~ n {\alpha^{n-1}}\, 
	\left(2 + \beta V(x) +  \beta \frac{K}{1-\gamma}  \right) K_f ~|\theta-\theta^\prime|  \nonumber \\
	&+  L_P~ {\alpha^{n}} ~ \frac{1}{\-1-\alpha}\left(1 + \beta \frac{K}{1-\gamma}\right) K_f ~|\theta-\theta^\prime| 	\label{eq:Delta n f psi le}.
	\end{align}

    Summation over $n,$ in view of \eqref{eq:Two parameter Poisson Lipschitz of u wrt theta D1}, yields the upper bound
	\begin{align}
	\label{eq:Two parameter Poisson Lipschitz of u wrt theta D2}
	& |u_{\theta, \psi}(x) - u_{\theta^{\prime}, \psi} (x)| \nonumber \\
	& \le L_P ~\frac{1}{(1-\alpha)^2} \left(2 + \beta V(x)  +  \beta \frac{K}{1-\gamma}\right) K_f ~|\theta-\theta^\prime| \nonumber\\
	& + L_P  ~\frac{1}{(1-\alpha)^2} \left(1 + \beta \frac{K}{1-\gamma}\right)K_f~|\theta-\theta^\prime|.
	\end{align}
    Combining \eqref{eq:Two parameter Poisson Lipschitz of u wrt theta D2} and \eqref{eq:Two parameter Poisson Lipschitz of u wrt psi} the proof is complete.
	\end{proof}
    \begin{remark}
    \label{remark_LU_calculations}
    The Lipschitz constant $L_u$ can be chosen as
        \begin{equation} 
          \label{eq:Poisson Lipschitz constant for u}
          L_{u} = \,{\frac 1 {1 - \alpha}}\! \left( L_f +  L_P \frac{2}{(1-\alpha)} K_f\right)\! \left(2 +  \beta \frac{K}{1-\gamma}\right)\!.
        \end{equation} 
    The details of this elementary calculation are omitted. 
    \end{remark}

\section{Relaxations of the Conditions}
\label{sec:Relaxations of the uniform drift condition}

In this section we restate the results of \cite{hairer2011yet} cited in Section \ref{sec:Introduction to HM} as Propositions \ref{prop:Beta Norm Contraction by P}--\ref{prop:Invariant Measure} under relaxed conditions, obtained by 
additional arguments. To facilitate reading and to highlight the parallel structure we state them as Propositions \ref{prop:Beta Norm Quasi Contraction by Pr}--\ref{prop:Invariant Measure Pr}.

A key condition of Propositions \ref{prop:Beta Norm Contraction by P} -- \ref{prop:Invariant Measure} is Assumption \ref{cond:Uniform Drift Condition}, requiring the existence of a common Lyapunov function. To illustrate the delicacy of this assumption consider the analogous scenario that a set of $n \times n$ matrices $\{A_{\theta}: \theta \in \Theta \subset \mathbf R^p\},$ such that $A_{\theta}$ is stable for all $\theta \in \Theta ,$ has a common quadratic Lyapunov function $V(x) := x^{\top} Q x$, where $Q$ is a symmetric positive definite matrix. In fact we would require that for some $0< \gamma < 1$ we have $
A_{\theta}^{\top} Q A_{\theta} \leq \gamma$ for all $ \theta \in \Theta$ in the sense of semi-definite ordering. Hence, the matrix $Q$ induces a metric with respect to which $A_{\theta}$ is a contraction, simultaneously for all $\theta$, with the same contraction factor $\gamma$, and thus the family of matrices $\{A_{\theta}: \theta \in \Theta$\} is jointly stable. 

Now, if joint stability fails to hold, but $\{A_\theta: \theta \in \Theta\}$ is a compact set of stable $n \times n$ matrices, then
we can find a positive integer $r$ such that $\Vert A_\theta^r \Vert \le \gamma_r <1$ for all $\theta \in \Theta,$ hence the family of matrices $\{A_\theta^r: \theta \in \Theta\}$ is jointly stable. This analogy   motivates the following relaxation of the drift condition given as Assumption \ref{cond:Uniform Drift Condition}, similar to  Assumption A'\!.5, (i) and (i') on page 290 of \cite{benveniste1990}:

\begin{assumption}[Uniform Drift Condition for $P_\theta^r$]
			\label{cond:Uniform Drift Condition for Pr theta}
			There exists a positive integer $r$, a measurable function $V:\mathbf{X}\rightarrow [0,\infty)$ and constants $\gamma_r\in(0,1)$ and $K_r\geq0$ such that for all  $\theta\in\Theta$ and $x\in{\mathbf{X}},$ we have
			\begin{equation}
			(\P^{*r}_{\theta} V)(x)\,\leq\, \gamma_r V(x)+K_r. 
			\label{eq:Drift Condition Pr}
			\end{equation}
	\end{assumption}			

\begin{assumption}[Uniform One Step Growth Condition for $P_\theta$]
	\label{cond:Uniform One Step Growth Condition for P theta}		With the same measurable function $V:\mathbf{X}\rightarrow [0,\infty)$ as above we have for all $\theta \in \Theta$ and all $x\in{\mathbf{X}}:$
			\begin{equation}
			(\P_{\theta}^{\ast} V)(x)\,\leq\, \gamma_1 V(x)+K_1,
			\label{eq:Uniform One Step Stab Condition}
			\end{equation}
	where we can and will assume that $\gamma_1 > 1$ and $K_1 \ge 0.$
\end{assumption}

Integrating \eqref{eq:Uniform One Step Stab Condition} with respect to any $\mu \in {\cal M}_V$ we get 
\begin{equation}
\P_{\theta} \mu (V)\, \leq\, \gamma_1 \mu (V) + K_1 \mu(\mathbf X),
\label{eq:Uniform One Step Stab Condition w mu}
\end{equation}
for all $\theta \in \Theta$, in exact analogy with \eqref{eq:Uniform Drift Condition w mu}.
The following 
lemma is a relaxed version of Proposition \ref{prop:Beta Norm Contraction by P}: 

\begin{lemma}
\label{lem:One Boundedness of P in 2 and 3 Norm} The uniform one-step growth condition given above implies that for {\it any} $\beta>0,$ for all functions $\varphi \in {\cal L}_V$ and all $\theta\in\Theta,$  with $\alpha^{\prime} = \gamma_1 \vee (1 + \beta K_1),$ we have 
		\begin{equation}
		\label{eq:One Step Stability of P star in 2 and 3 Norm}
		\Vert {\P_{\theta}^{\ast} \varphi} \Vert_\beta \,\leq\, \alpha^{\prime}\Vert {\varphi} \Vert_\beta \quad \text{\it and} \quad 
\vertiii{\P_{\theta}^{\ast} \varphi}_\beta \,\leq\, \alpha^{\prime} \vertiii{\varphi}_\beta.
		\end{equation}
\end{lemma}
The first statement is obtained by straightforward calculations, while the second statement follows by using the definition 
$\vertiii{\psi}_\beta = \min_c \| \psi + c \|_\beta.$ For details see the Appendix.
\begin{lemma}
\label{lemma:one step Stab for sigma beta}
Under Assumption \ref{cond:Uniform One Step Growth Condition for P theta}, for any $\beta>0,$ for all $\theta \in \Theta$, the kernel $P_\theta$ is a bounded linear operator on  ${\cal M}_{V}^0$, more exactly, for any $\eta \in {\cal M}_{V}^0$, we have, with $\alpha'$ as in Lemma \ref{lem:One Boundedness of P in 2 and 3 Norm},
\begin{equation}
\label{eq:Contraction of sigma beta with eta CP}
\sigma_\beta(\P_\theta \eta)\,\leq\, \alpha' \sigma_\beta(\eta).
\end{equation}
Alternatively, let $\mu_1,\mu_2$ be two possibly signed measures on $\mathbf X$ as in Definition \ref{def:sigma beta for pairs}. Then we have
	\begin{equation}
	\label{eq:one step Stab for sigma beta}
	\sigma_\beta(\P_{\theta}\mu_1,\P_{\theta}\mu_2)\, \leq\, \alpha^\prime \sigma_\beta(\mu_1,\mu_2).
	\end{equation}
\end{lemma}

\begin{assumption}[Uniform Local Minorization for 
  $P_{\theta}^r$]
  \label{cond:Local Minorization Extended to Pr Uniform}
  Let $R_r>2K_r/(1-\gamma_r),$ where $\gamma_r$ and $K_r$ are the constants from Assumption \ref{cond:Uniform Drift Condition for Pr theta}, and let ${\cal C}_r=\{x\in{\mathbf X} :  V(x) \leq R_r \}$.
  There exist a probability measure $\bar{\mu}_r$ and a constant $\bar{\alpha}_r\in(0,1)$ such that for all $\theta \in \Theta$, $x \in {\cal C}_r$ and measurable $A$ it holds:
  \begin{equation}
  \label{eq:Local Minorization Extended to Pr Uniform}
    \P_\theta^r(x,A)\,\geq\,\bar{\alpha}_r \bar{\mu}_r(A).
  \end{equation}
\end{assumption}

The first main result of \cite{hairer2011yet}, cited as Proposition \ref{prop:Beta Norm Contraction by P} in Section \ref{sec:Introduction to HM}, can be restated as follows:
\begin{proposition}
\label{prop:Beta Norm Quasi Contraction by Pr}
Under Assumptions \ref{cond:Uniform Drift Condition for Pr theta},  \ref{cond:Uniform One Step Growth Condition for P theta} and \ref{cond:Local Minorization Extended to Pr Uniform} there exist constants $\beta >0$, $\alpha  \in(0,1)$ and $C >0$ such that for any $\varphi \in {\cal L}_V,$ all $\theta \in \Theta$ and all $n>0$ we have:
	\begin{equation}
	\vertiii{\P_\theta^{*n} \varphi}_\beta \,\leq\, C \alpha^n \vertiii{\varphi}_\beta.
	\end{equation}
Here we can choose $\beta = \beta_r$ and $\alpha= \alpha_r^{1/r}$, with $\beta_r$ and $\alpha_r$ provided by Proposition \ref{prop:Beta Norm Contraction by P} applied to $\P_\theta^r$, and noting that $0 < \alpha_r <1,$ and $C = \alpha_r^{-1}(\alpha^{\prime})^{r-1},$ with $\alpha'$ as in Lemma \ref{lem:One Boundedness of P in 2 and 3 Norm}.
\end{proposition}

\begin{proof} Let us fix a $\theta \in \Theta$ and write $P = \P_\theta$. By Proposition \ref{prop:Beta Norm Contraction by P} there exist 
$\beta = \beta_r >0,$ and $\alpha_r\in(0,1)$ such that $\vertiii{\P^{\ast r} \varphi}_\beta \leq \alpha_r \vertiii{\varphi}_\beta,$ implying for any positive integer $m$
	\begin{equation}
	\vertiii{\P^{\ast rm} \varphi}_\beta \leq \alpha_r^m \vertiii{\varphi}_\beta. 
	\end{equation}
	For a general positive integer $n$ write $n=rm +k$ with $0 \le k \le r-1$. Then, we  get  
	\begin{equation}
	\vertiii{\P^{\ast n} \varphi}_\beta \leq \alpha_r^m \vertiii{\P^{\ast k}\varphi}_\beta.
	\end{equation}
To complete the proof estimate the last term above  applying the second inequality of \eqref{eq:One Step Stability of P star in 2 and 3 Norm} $k \le r-1$ times to obtain  
\begin{equation}
\vertiii{\P^{\ast n} \varphi}_\beta \leq \alpha_r^m (\alpha^{\prime})^{r-1} \vertiii{\varphi}_\beta. 
\end{equation}
Now $m = (n-k)/r > n/r - 1,$ hence $\alpha_r^m < \alpha_r^{n/r} \alpha_r^{-1},$ thus the claim of the theorem follows.
\end{proof}

Repeating the arguments leading to Proposition \ref{prop:Contraction of sigma beta}, we get:

\begin{proposition}
    	\label{prop:Quasi Contraction of Pr wrt sigma beta} Under Assumptions \ref{cond:Uniform Drift Condition for Pr theta},  \ref{cond:Uniform One Step Growth Condition for P theta} and \ref{cond:Local Minorization Extended to Pr Uniform} there exist constants $\beta = \beta_r >0,$ $\alpha\in(0,1)$ and $C >0$ such that for any signed measure $\eta \in {\cal M}_{V}^0$, all $\theta$ and $n \ge 0$ we have
	\begin{equation}
	\label{eq:Quasi Contraction of Pr with eta wrt sigma beta}
	\sigma_\beta(\P_\theta^n \eta)\,\leq\, C \alpha^n \sigma_\beta(\eta).
	\end{equation}
The constants $\beta = \beta_r,$ $\alpha$ and $C$ are the same as in Proposition \ref{prop:Beta Norm Quasi Contraction by Pr}. Alternatively, let $\mu_1,\mu_2$ be two possibly signed measures on $\mathbf X$ as in Definition \ref{def:sigma beta for pairs}. Then we have
	\begin{equation}
	\label{eq:Quasi Contraction of Pr wrt sigma beta}
	\sigma_\beta(\P_\theta^n \mu_1,\P_\theta^n \mu_2)\,\leq\, C \alpha^n \sigma_\beta( \mu_1, \mu_2).
	\end{equation}
\end{proposition}

Finally, we have the following extension of Proposition \ref{prop:Invariant Measure}:

\begin{proposition} 
	\label{prop:Invariant Measure Pr}
	Under Assumptions \ref{cond:Uniform Drift Condition for Pr theta},  \ref{cond:Uniform One Step Growth Condition for P theta} and \ref{cond:Local Minorization Extended to Pr Uniform} for all $\theta \in \Theta$ there exists a unique probability measure $\mu_\theta^\ast$ on $\mathbf X$ such that $\mu_\theta^\ast(V) < \infty$ 	and $\P_\theta \mu_\theta^\ast= \mu_\theta^\ast.$ Denoting the unique invariant probability measure for $\P_\theta^r$ by 
$\mu_{\theta,r}^{\ast}$ we have $\mu_{\theta}^{\ast}=\mu_{\theta,r}^{\ast}.$
\end{proposition}

\begin{proof}  Let us fix any $\theta \in \Theta$ and write $\P = \P_\theta$, $\mu^\ast = \mu_\theta^\ast$ and $\mu_{r}^{\ast} = \mu_{\theta,r}^{\ast}.$ Thus  $\mu^{\ast}_{r}$ is the unique invariant probability  measure for $\P^r$ the existence of which is ensured by Proposition \ref{prop:Invariant Measure}. Now, we show that for any $k>0$ we have $\int_{\mathbf X} V  \, \mathrm{d} \P^k \mu^{\ast}_{r}: = \int_{\mathbf X} V(x)  \, \P^k \mu^{\ast}_{r} (\mathrm{d} x) < \infty$. Indeed, write:
\begin{equation}
\int_{\mathbf X} V  \, \mathrm{d} \P^k \mu^{\ast}_{r} =  
\int_{\mathbf X} (\P^k)^\ast V \,
 \mathrm{d}  \mu^{\ast}_{r} =
\int_{\mathbf X} (\P^\ast)^k V 
\, \mathrm{d}  \mu^{\ast}_{r}.
\end{equation}
Here the r.h.s. can be bounded from above, using the definition of $\| \cdot \|_{\beta}$ and the first half of \eqref{eq:One Step Stability of P star in 2 and 3 Norm}, by   
\begin{equation}
\int_{\mathbf X} \Vert (\P^\ast)^k V \Vert_\beta (1 + \beta V) \, \mathrm{d}  \mu^{\ast}_{r} \le  
\int_{\mathbf X} (\alpha^\prime)^k (1 + \beta V) 
\, \mathrm{d} \mu^{\ast}_{r},
\end{equation}
which is finite since $\int_{\mathbf X} {V (x)\, \mathrm{d}\mu^{\ast}_{r} (x)} <\infty.$
It follows that the probability measure $\mu^\ast$ defined by
\begin{equation}
\mu^\ast := \frac 1 r (I + \P + \ldots \P^{r-1}) \mu^{\ast}_{r}
\end{equation}
also satisfies $\int_{\mathbf X} {V (x) \, \mathrm{d}\mu^\ast (x)} <\infty,$ and it is readily seen to be invariant for $\P.$ Since any probability measure invariant for $\P$ is also invariant for $\P^r,$ there cannot be measures that are invariant for $\P$ besides $\mu^{\ast}_{r}$, and thus we \tb{have} $\mu^\ast = \mu^{\ast}_{r}$.
\end{proof}

\section{Analysis of the Poisson Equation under Relaxed Conditions}
\label{sec:Poisson Equations under Relaxed Conditions}

In what follows, the main results of Sections \ref{sec:Poisson Equation - Existence UniquenessMain results} and \ref{sec:Poisson Equation - Lipschitz Continuity} will be now extended, with minor modifications, assuming the above relaxed conditions. 
\smallskip

\begin{theorem}
  \label{th:Poisson Eq Existence Pr} 
Let Assumptions \ref{cond:Uniform Drift Condition for Pr theta}, \ref{cond:Uniform One Step Growth Condition for P theta} and \ref{cond:Local Minorization Extended to Pr Uniform} hold. 
Let $\beta = \beta_r > 0$ be as given in Proposition \ref{prop:Beta Norm Quasi Contraction by Pr}. Let $f: {\mathbf X} \rightarrow \mathbb{R}$ be a measurable function 
such that $ \vertiii{f}_\beta<\infty$. Let $\P=\P_\theta$ for some fixed $\theta$, and 
let $\mu^\ast$ denote the unique invariant probability measure of $\P,$ and let $h=\mu^\ast(f).$ 
Then, the Poisson equation
  \begin{equation}
    (I-\Ps)u(x) = f(x) - h
    \label{eq:Poisson CP}
  \end{equation}
  has a unique solution $u(\cdot)$ up to additive constants. The particular solution for which $\mu^\ast(u) = 0$ can be written as
   \begin{equation}
   u(x) = \sum_{n=0}^{\infty} (P^{\ast n}  f(x) - h),
   \label{eq:Poisson Solution u as Inf Sum for Pr}
   \end{equation}
where the right hand side is absolutely convergent, and  
   \begin{equation}
   \label{eq:Poisson Eq Upper Bound for ux with V(x) Pr}
    |u(x)|\,\leq\, \vertiii{f}_{\beta} K_{r,u}(1+ \beta V(x)) 
    \end{equation}
  for some constant $K_{r,u}>0$ depending only on the constants appearing in Assumptions \ref{cond:Uniform Drift Condition for Pr theta}, \ref{cond:Uniform One Step Growth Condition for P theta} and \ref{cond:Local Minorization Extended to Pr Uniform}. It also  follows:
  \begin{equation}
  \|u\|_\beta \,\leq\,  K_{r,u} \vertiii{f}_\beta < \infty.
  \end{equation}  
  \end{theorem}
\begin{proof} Consider the Poisson equation
		\begin{equation}
		\label{eq:Poisson for Pr}
		(I-\P^{*r})v(x) = f(x) - h,
		\end{equation}
		where $h= \mu^{\ast} (f),$ recalling that $\mu^{\ast}_r = \mu^{\ast}.$ In view of Theorem \ref{th:Poisson Eq Existence} it has a unique solution, up to an additive constant. The particular solution with  $\mu^\ast(v) = 0$ can be written as 
		\begin{equation}
		v(x) = \sum_{n=0}^{\infty} (P^{\ast nr}  f(x) - h),
		\label{eq:Poisson Solution u as Inf Sum CP}
		\end{equation}
		which is well-defined, the r.h.s.\ is absolutely convergent, and
			\begin{equation}
			|v(x)|\leq  \vertiii{f}_\beta \hspace{0.4mm} K_{r,v} (1 + \beta V(x)), 
			\label{eq:Poisson eq vx upper bound Pr}
			\end{equation}
			implying the inequality 
			\begin{equation}
                          \label{eq:Poisson eq norm upper bound for growth of v Krv}
                          \|v\|_\beta\!\leq \!K_{r,v} \vertiii{f}_\beta~,
			\end{equation}
			where $K_{r,v}$ is given by 
			\begin{equation}
			\label{eq:K_rv_definition}
			K_{r,v} := \frac{1}{1-\alpha_r} \max \left( 2 + \beta \frac{K_r}{1-\gamma_r}, \beta \right).
			\end{equation}
\noindent	The solution of \eqref{eq:Poisson for Pr} is related to that of \eqref{eq:Poisson CP} by noting that
		\begin{equation}
		(I-\P^{\ast r}) = 	(I-\P^{\ast}) (I + \P^{\ast} + \ldots + \P^{\ast (r-1)}).  
		\label{eq:Poisson for Pr Decomposed}
		\end{equation}
		It follows that  
		\begin{align} 
		\label{eq:Poisson for P vs Pr}
		u(x) := (I + \P^{\ast} + \ldots + \P^{\ast (r-1)}) v(x) 
		\end{align}
		is a solution of \eqref{eq:Poisson CP}.
  To get an upper bound for $|u(x)|$, write 
		\begin{align} 
		\label{eq:Poisson for P vs Pr V2}
		u(x) = (I + \P + \ldots+ \P^{r-1}) ~\delta_x (v). 
		\end{align}
		Taking into account the upper bound for $|v(x)|$ given in \eqref{eq:Poisson eq vx upper bound Pr}, it is seen that it is sufficient to derive upper bounds for $(P^k \delta_x)(V) = P^{\ast k} V (x)$ for $k=1, \ldots , r-1.$

Now, in view of the one-step growth condition and inequality \eqref{eq:Uniform One Step Stab Condition w mu}, we have $\P \mu (V) \leq \gamma_1 \mu (V) + K_1$, for any probability measure $\mu$ \tb{with} $\mu(V) <\infty$. 
		By repeated application of this inequality we obtain the upper bound for $(\P^k\mu)(V)$: 
		\begin{align}
		\gamma_1^k \mu(V) + \sum_{\ell=0}^{k-1}\gamma_1^\ell K_1 \le  
        \gamma_1^k \mu(V)+\frac{\gamma_1^k K_1}{\gamma_1-1}.
        \label{eq:Pk mu V upper bound under one step growth}
		\end{align}

		Using $\mu = \delta_x$ and summing over $k$ from $0$ to $r-1$ we get 
		\begin{align*}
		(I + \P + \ldots + \P^{r-1}) \delta_x (V)  \le \sum_{k=0}^{r-1} \gamma_1^k \left(V(x)  + \frac{K_1}{\gamma_1-1} \right). 
		\end{align*}
		The right hand side is bounded from above by  
		\begin{equation}
		\frac{\gamma_1^r }{\gamma_1-1} \left(V(x)  + \frac{K_1}{\gamma_1-1} \right). 
		\end{equation}
		Combining these inequalities with \eqref{eq:Poisson eq vx upper bound Pr} and \eqref{eq:Poisson for P vs Pr} we get   
		\begin{align} 
		 |u (x)| = &|(I + \P + \ldots+ \P^{r-1}) 
		\delta_x (v)| \le \vertiii{f}_{\beta} ~K_{r,v} \times\nonumber \\ 
		& \times\! \left (r+ \beta  \frac{\gamma_1^r }{\gamma_1-1} \left(V(x)  + \frac{K_1}{\gamma_1-1} \right)   \right),
		\end{align}
		implying the upper bound of the form given in \eqref{eq:Poisson Eq Upper Bound for ux with V(x) Pr}.

		As for {\it uniqueness}, assume that there are two solutions $u_1, u_2 \in \mathcal{L}_V$, and let $\Delta u=u_2-u_1$. 		
Then, $(I-\Ps) \Delta u(x)=0$ for all $x$, implying $\Ps \Delta u=  \Delta u$. Iterating this $r-1$ times we get $\P^{*r} \Delta u=\Delta u,$ and by Theorem  \ref{th:Poisson Eq Existence} we conclude that $\Delta u$ is a constant function, thus completing the proof. 
\end{proof}

The extension of Theorem \ref{thm:Poisson Lipschitz}, on the Lipschitz continuity of $u_\theta(\cdot)$, seems to be straightforward. However, we should point out that 
we have to assume the  Lipschitz continuity of the {\it one-step} kernels 
$(\P_\theta),$ as given in Assumption \ref{cond:Lipschitz of P theta Dirac x by Miklos}. 

\begin{theorem} 
\label{th:Poisson Lipschitz w Uniform cond on Pr}
Assume that the kernels $(\P_\theta)$ satisfy Assumptions \ref{cond:Uniform Drift Condition for Pr theta}, \ref{cond:Uniform One Step Growth Condition for P theta} and \ref{cond:Local Minorization Extended to Pr Uniform}. In addition assume that the family of one-step kernels $(\P_\theta)$ is Lipschitz continuous is the sense of Assumption \ref{cond:Lipschitz of P theta Dirac x by Miklos}.
Let us fix $\beta= \beta_r > 0$ as given in Proposition \ref{prop:Beta Norm Quasi Contraction by Pr}. Let $(f_\theta)$ be a family of  ${\mathbf X} \rightarrow \mathbb{R}$ measurable functions such that Assumption \ref{cond:f theta beta Lipschitz} holds with the above $\beta$. 
Let $\mu_\theta^\ast$ denote the unique invariant probability measure of $\P_\theta$ and let $h_\theta=\mu_\theta^\ast(f_\theta).$ Consider the  Poisson equations
				\begin{equation}
				\label{eq:Poisson w theta w Uniform cond on Pr}
				(I-P_\theta^\ast)u_\theta(x) = f_\theta(x) - h_\theta. 
				\end{equation}
				Then, $h_\theta$ is Lipschitz continuous in $\theta$:
				\begin{equation}
				\label{eq:h Lipshitz w Uniform cond on Pr}
				|h_\theta-h_{\theta'}|\,\leq\, L_{r,h} |\theta-\theta'|, \\
				\end{equation}
				and the particular solution, given in Theorem \ref{th:Poisson Eq Existence Pr} by \eqref{eq:Poisson Solution u as Inf Sum for Pr} as $
				u_\theta(x)=\sum_{n=0}^{\infty} (P_\theta^{\ast n}  f_\theta(x) - h_\theta)$ 			
				is Lipschitz continuous in $\theta$:
				\begin{equation}
				\label{eq:u Lipshitz w Uniform cond on Pr}
				|u_\theta(x)-u_{\theta'}(x)|\,\leq\, L_{r,u} |\theta-\theta'|(1+ \beta V(x))
				\end{equation}
where $L_{r,u}$ is independent of $x$. Alternatively, we can write 
				\begin{equation}
				\label{eq:u Lipshitz w Uniform cond on Pr b}
				\|u_\theta-u_{\theta'}\|_{\beta}\, \leq\, L_{r,u} |\theta-\theta'|.
				\end{equation}
				Here the constants $L_{r,h}$ and $L_{r,u}$ depend only on the constants appearing in Assumptions \ref{cond:Lipschitz of P theta Dirac x by Miklos}, \ref{cond:f theta beta Lipschitz}, \ref{cond:Uniform Drift Condition for Pr theta}, \ref{cond:Uniform One Step Growth Condition for P theta}, and \ref{cond:Local Minorization Extended to Pr Uniform}.
				\end{theorem}

For the proof we need a simple variant of Lemma \ref{lemma:Lipschitz for Pn with eta} providing an at most exponentially growing bound for $\sigma_\beta(\P^n_{\theta}\eta,\P^n_{\theta^{\prime}})$:
\begin{lemma} 
\label{lemma:Lipschitz for Pn with OneStepGrowth}
Let $(\P_\theta)$ satisfy the uniform one-step growth condition, Assumption \ref{cond:Uniform One Step Growth Condition for P theta}.  In addition, let Assumption \ref{cond:Lipschitz of P theta Dirac x by Miklos}, requiring the Lipschitz continuity of $(\P_\theta),$ hold with some $\beta >0.$  Then for any signed measure $\eta$ with $|\eta|(1+\beta V) < \infty$ and $\theta,\theta^{\prime}\in\Theta,$ and for any $\alpha^{\prime\prime} > \alpha^\prime : = \max (1 + \beta K_1, \gamma_1)$ 
\begin{equation}
  \label{eq:Lipschitz for Pn with any mu Pr}
  \sigma_\beta(\P^n_{\theta}\eta,\P^n_{\theta^{\prime}}\eta)\, \le\, L_P'' |\theta-\theta^{\prime}| (\alpha^{\prime\prime})^{n} 
  |\eta|\left(1+\beta V\right)
\end{equation}
for all $n>0$, where $L_P''$ depends only on 
the constants appearing in the conditions of the lemma
and on $\alpha^{\prime\prime}$.
\end{lemma}
The proof follows the proof of Lemma \ref{lemma:Lipschitz for Pn with eta}, however to estimate $|\P_{\theta^{\prime}}^k\eta|(V)$, we need a modification of \eqref{eq:Pk mu V upper bound under one step growth}, with $|\eta|$ replacing $\mu,$ yielding a restatement of \eqref{eq:Pk mu V upper bound under one step growth CP}:\vspace{-1mm}
  \begin{align}
  \label{eq:One Step Pk eta V upper bound growth} 
    |\P_{\theta'}^k\eta|(V) \leq  \gamma_1^k \left(|\eta|(V)+K_1|\eta|(\mathbf X)/(\gamma_1-1) \right).
  \end{align}
Details will be given in the Appendix.

			\begin{proof}[Proof of Theorem \ref{th:Poisson Lipschitz w Uniform cond on Pr}]
			First, note that $\mu_{\theta}^{*} = \mu_{\theta,r}^{*}$ implies that
			\begin{equation}
			h_{\theta} \,=\, \mu_{\theta}^{*}(f_{\theta})\, = \,  \mu_{\theta,r}^{*} (f_{\theta}).
			\end{equation}
			Applying Theorem \ref{thm:Poisson Lipschitz} for the Poisson equation
			\begin{equation}
			\label{eq:Poisson for Pr CP}
			(I-P_\theta^{* r})v_\theta(x) \,=\, f_\theta(x) - h_\theta,
			\end{equation}
			noting that $P_\theta^{* r}$ satisfies the relevant conditions in view of Lemma \ref{lemma:Lipschitz for Pn with eta}, we conclude that $h_{\theta}$ is Lipschitz continuous in $\theta$:
			\begin{equation}
			|h_\theta-h_{\theta'}|\leq L_{r,h} |\theta-\theta'|, \\
			\label{eq:Poisson h theta psi Lipschitz Pr}
			\end{equation}
			where $L_{r,h}$ is given, according to \eqref{eq:Poisson Lipschitz constant for h}, by 
\begin{equation}
L_{r,h} = \left(L_f +  K_f L_{P^r} \frac{1}{\-1-\alpha_r} \right)\left(1 + \beta_r \frac{K_r}{1-\gamma_r}\right)\!,
\end{equation}
where $L_{P^r}$ is the Lipschitz-constant for $(P^r_\theta),$ ensured by Lemma \ref{lemma:Lipschitz for Pn with eta}, and $\beta_r$ and $\alpha_r$ are chosen as in Proposition \ref{prop:Beta Norm Quasi Contraction by Pr}.

{In order to prove the second part of Theorem \ref{th:Poisson Lipschitz w Uniform cond on Pr}, note  that, in view of Theorem \ref{thm:Poisson Lipschitz}, the particular solution given by 
   $$
   v_\theta(x)=\sum_{n=0}^{\infty} P_\theta^{* nr}  (f_\theta(x) - h_\theta)
   $$
   is Lipschitz continuous w.r.t.\ $\theta$, and
				\begin{equation}
				\label{eq:v Lipshitz w Uniform cond on Pr}
				|v_\theta(x)-v_{\theta'}(x)|\leq L_{r,v} |\theta-\theta'| (1+ \beta V(x)),
				\end{equation}
				where $L_{r,v}$ is defined in \eqref{eq:Poisson Lipschitz constant for u} in the role of $L_u$, according to 
				\begin{equation*}
				L_{r,v} = {\frac 1 {1 - \alpha_r}} \left( L_f +  L_{P^r} \frac{2}{(1-\alpha_r)} K_f\right) \left(2 +  \beta_r \frac{K_r}{1-\gamma_r}\right).
				\end{equation*}
				Hence the particular solution of the Poisson equation \eqref{eq:Poisson w theta w Uniform cond on Pr}: 
				\begin{equation} 
				\label{eq:Possion_solution_for_Pr_rep}
				u_\theta(x) := (I + \P_\theta^\ast + \ldots + \P_\theta^{* (r-1)}) v_\theta(x)
				\end{equation}
				is also Lipschitz continuous in $\theta.$ Indeed, for $1 \le m \le r-1$
				\begin{equation}
				\label{eq:Lip_P_theta_m_decomposed}
				\begin{aligned}
				& P_\theta^{* m} v_\theta - P_{\theta^\prime}^{* m} v_{\theta^{\prime}}\\
				=\; & P_\theta^{* m} (v_\theta -v_{\theta^\prime} ) + (P_{\theta}^{\ast m} - P_{\theta^\prime}^{\ast m}) v_{\theta^{\prime}},
				\end{aligned}
				\end{equation}
				and the first term on the r.h.s. is bounded from above as 
				\begin{equation*}
				| P_\theta^{\ast m} (v_\theta -v_{\theta^\prime} )(x) | \leq L_{r,v} |\theta-\theta'| P_\theta^{\ast m}(1+ \beta V)(x),
				\end{equation*}
				for all $x$. 
Applying \eqref{eq:One Step Pk eta V upper bound growth}  
    with $\eta = \delta_x$ we get the upper bound
				\begin{equation}
				\label{eq:Lip_P_theta_m_decomposed_term1_upper_bound}
				|P_\theta^{\ast m} (v_\theta -v_{\theta^\prime} )| \le L_{r,v} |\theta-\theta'| \gamma_1^m \left(V(x)+\frac{ K_1}{\gamma_1-1} \right).
				\end{equation}
				
For the second term on the right hand side of \eqref{eq:Lip_P_theta_m_decomposed} first we note that $\vertiii{v_{\theta^\prime}}_{\beta} < \infty$ by \eqref{eq:Poisson eq norm upper bound for growth of v Krv} in the proof of Theorem \ref{th:Poisson Eq Existence Pr}:
				\begin{equation}
				\vertiii{v_{\theta^\prime}}_{\beta} \le \| {v_{\theta^\prime}}\|_{\beta} \leq K_{r,v} \vertiii{f_{\theta^\prime}}_\beta \le  K_{r,v} K_f
				\end{equation}
				for some constant $K_{r,v}>0$ depending only on the constants appearing in Assumptions \ref{cond:Uniform Drift Condition for Pr theta}, \ref{cond:Uniform One Step Growth Condition for P theta} and \ref{cond:Local Minorization Extended to Pr Uniform}. Now we can write 
				\begin{align}
				|(P_{\theta}^{\ast m} - P_{\theta^\prime}^{\ast m}) v_{\theta^{\prime}}(x)| &=
				|(P_{\theta}^{m} - P_{\theta^\prime}^{m})\delta_x(v_{\theta^{\prime}})| \nonumber\\
				&\leq \sigma_{\beta}\big((P_{\theta}^{m} - P_{\theta^\prime}^{m})\delta_x\big) \hspace{0.4mm}\vertiii{v_{\theta^{\prime}}}_{\beta}\nonumber\\
				 &= 
				\sigma_{\beta}(P_{\theta}^{m}\delta_x, P_{\theta^\prime}^{m}\delta_x)\hspace{0.5mm}\vertiii{v_{\theta^{\prime}}}_{\beta}.
				\label{eq:P theta m minus P theta prime m times v}
				\end{align}
                Applying Lemma \ref{lemma:Lipschitz for Pn with OneStepGrowth} we get that 
    the r.h.s. is bounded by
				\begin{equation}
				\label{eq:Lip_P_theta_m_decomposed_term2_upper_bound}
				L_P'' |\theta-\theta^{\prime}|  (\alpha^{\prime\prime})^{m}  \left(1+\beta V(x)\right) \cdot K_{r,v}  \vertiii{f_{\theta^{\prime}}}_\beta \hspace{0.4mm} 
				\end{equation}
				where $K_{r,v}$ is defined by \eqref{eq:K_rv_definition}. Recall that $\sup_{\theta^{\prime}} \vertiii{f_{\theta^{\prime}}}_\beta = K_f < \infty$ by  Assumption \ref{cond:f theta beta Lipschitz}. Taking into account the representation of $u_\theta(x)$ given in  \eqref{eq:Possion_solution_for_Pr_rep}, and the decomposition given in \eqref{eq:Lip_P_theta_m_decomposed}, and adding the upper bounds 	\eqref{eq:Lip_P_theta_m_decomposed_term1_upper_bound} and  \eqref{eq:Lip_P_theta_m_decomposed_term2_upper_bound} for $m= 0, \dots, r-1$, we get the second claim.	}
\end{proof}

\section{Design of Queuing Systems}
\label{sec:Design of Quees}

\ms
In this section the applicability of our results will be demonstrated on the modification of a classical textbook example, a simple queuing system. For the sake of simplicity we restrict attention to a single arrival process handled by a single server, where both the arrival and service process may be subject to control. For now we consider a system with open loop control, a multivariate version of which may be a realistic model for a client assignment systems. Extension of our approach to queuing systems allowing feedback, such as Call Admission Control, see e.g., Chapter 11 of \cite{queue}, may be the subject of future research.

Let us first describe the dynamics of a queue without control. Let the arrival process be identified by a simple point process ${\tau_n}, \, n \ge 0$ with $\tau_0=0,$ and let $T_{n+1}$ be the (finite) time elapsed between the arrivals of customers $n$ and $n+1,$ i.e. $T_{n+1} =  \tau_{n+1} - \tau_n.$ Let $S_{n}$ be the service time of customer $n.$ It is assumed that \tb{the sequences $(T_{n})$ and $(S_{n})$, $n\in\mathbb{N}$,} are i.i.d.\
sequences of $\mathbb{R}_{+}$-valued random variables, respectively, independent of each other. We define $x^{+}:=\max\{x,0\}$, $x\in\mathbb{R}$.
The waiting time of the $n$-th customer will be denoted by $W_{n}.$ It is readily seen that it satisfies the recursion, with $W_{0}=0$ as initial value,  $n\in\mathbb{N}$ and with $ U_{n+1} = S_{n} - T_{n+1}$: 
\begin{equation}
\label{eq:Dynamics_of_Queue}
W_{n+1}=(W_{n} +  U_{n+1})^{+}. 
\end{equation}

\tb{In the case of controlled queues both the service time and the arrival time may depend on a control parameter $\theta$. Let $\Theta\subset\mathbb{R}^{k}$ be a connected, open set as above, and let $D \subset\mathbb{R}^{k}$ be a compact set such that $D \subset \Theta.$} The choice of the control parameter $\theta$ may determine the law of the service times and that of the arrival times.  Compactness of $D$ is a technical condition needed for the verification of Assumption  \ref{cond:Uniform Drift Condition for Pr theta}. The dynamics of the queue is described, with $W_{\theta, 0}=0,$  and $U_{\theta, n+1} := S_{\theta, n} - T_{\theta, n+1},$ by 
\begin{equation}
\label{eq:Dynamics_of_Queue_with_Parameter}
W_{\theta, n+1}  =(W_{\theta, n} + U_{\theta, n+1})^{+}.
\end{equation}
If the initial condition is $W_{\theta, 0}=x$, then the waiting time at $n$ will be denoted by 
$W_{\theta, n}(x).$ To guarantee stability of the queue we have to assume 
\begin{equation}
\mathbb E[\, U_{\theta, 1}\, ]< 0,
\end{equation}
for all $\theta\in D.$ This is a standard condition implying stability of the queue for any fixed $\theta$ in a variety of interpretations, see e.g., \cite {loynes1962,borovkov} and \cite{diaconis1999iterated}. 

The validity of the drift condition, given as Assumption \ref{cond:Uniform Drift Condition}, with no parameter-dependence, has been established under appropriate technical conditions, using the Lyapunov function $V(x):=e^{\chi x}$, $x\in\mathbb{R}_{+},$ with $\chi>0$ small enough, in \ Section 16.4 of \cite{meyn2012markov}. A uniform version of this result will be established below. In fact, we will show that under reasonable additional conditions on a controlled queue Assumptions \ref{cond:Uniform Drift Condition}, \ref{cond:Lipschitz of P theta Dirac x by Miklos}, \ref{cond:Local Minorization Extended to Pr Uniform} are satisfied with $V(x):=e^{\chi x}$, $x\in\mathbb{R}_{+},$ with $\chi>0$ small enough, when $\theta$ is restricted to compact set $D \subset \Theta$. Thus the results of the paper, in particular Theorems \ref{prop:Invariant Measure Pr}, \ref{th:Poisson Eq Existence Pr}, \ref{th:Poisson Lipschitz w Uniform cond on Pr}, imply the existence, uniqueness and Lipschitz continuity of the solution $u_\theta(x)$ of the parameter-dependent Poisson equation
	\begin{equation}
	\label{eq:Poisson eq with theta queue}
	u_\theta(x) - \mathbb E_\theta [\, W_{\theta, 1} \hspace{0.5mm}\vert\hspace{0.5mm} W_0=x\, ] \,=\, f_\theta(x) - h_\theta,
	\end{equation}
	where the normalizing constant $h_\theta$ is the expectation of $f_\theta(x)$ under the (unique) invariant measure, when $\theta$ is restricted to an open set $\Theta^\prime \subset D$ in place of $\Theta$.

The conditions below will be given in terms of the r.v. $U_{\theta, 1}$ thus ensuring the generality of our results. Specific conditions in terms of $S_{\theta, 0}$ and $T_{\theta, 1}$ will be given at the end of the section. 
To guarantee stability of the system \eqref{eq:Dynamics_of_Queue_with_Parameter} we stipulate:

\begin{assumption}
	\label{cond:Stability_of_Queue} We have $ \mathbb E[\, U_{\theta, 1}\, ]< 0$, for all $\theta\in D.$
\end{assumption}

This is a standard condition implying stability of the queue for any fixed $\theta$ in a variety of interpretations, see e.g., \cite {loynes1962,borovkov} and \cite{diaconis1999iterated}. 
A further standard condition in queuing theory, and also in the area of risk processes \cite{andersen}, is the existence of a finite positive exponential moment of $S_{\theta, n} - T_{\theta, n+1},$ or equivalently that of $U_{\theta, n}.$  A uniform version of this condition in terms of $U_{\theta, 1}$ is given below: 

\begin{assumption}
	\label{cond:Exponential_Moment_for_U_Uniform}
	\tb{We have $\sup_{\,\theta\in D}\mathbb E[\,\exp(\eta \, U_{\theta, 1})\, ]<\infty$, for some $\eta >0$.}
\end{assumption}

\tb{Observe that} Assumption \ref{cond:Exponential_Moment_for_U_Uniform}
is automatically satisfied if 	$\sup_{\,\theta\in D}\mathbb E[\,\exp(\eta \,  S_{\theta, 0})\, ]<\infty.$ Finally, we will need the following continuity condition for $U_{\theta, 1}$:
\begin{assumption}
	\label{cond:Weak_Cont_of_U_theta}  
	The probability distribution function of $U_{\theta, 1}$ is weakly continuous in $\theta$ for $\theta\in D,$ \tb{i.e. $\mathbb E[\, f(U_{\theta, 1})\, ]$ is continuous in $\theta$ for all bounded continuous functions $f$.}
\end{assumption}

We note in passing that these three assumptions imply that the stability condition, 
Assumption \ref{cond:Stability_of_Queue}, is satisfied uniformly in $\theta$ for $\theta \in D:$\vspace{-1mm}
\begin{equation}
\label{eq:Stability_of_Queue_Uniform}
\sup_{\theta\in D} \mathbb E[\, U_{\theta, 1}\, ]< 0.
\end{equation}

{\bf Uniform Drift Condition.} The validity of the uniform drift condition, given as Assumption \ref{cond:Uniform Drift Condition}, with no parameter-dependence, has been established using the Lyapunov function $V(x):=e^{\chi x}$, $x\in\mathbb{R}_{+},$ with $\chi$ small enough, 
see e.g., Section 16.4 of \cite{meyn2012markov}. (We should note that the use of exponential moments is also a standard tool in the theory of risk processes, see \cite{andersen}). For the sake of completeness and further reference we restate this result, and provide its proof.

\begin{lemma}
	\label{lem:Drift_Cond_Valid}
	Let us assume that $U$ is an $\mathbb R$-valued random variable such that $\mathbb E \, U < 0,$ and for some $\eta >0$ we have ${\mathbb E} \, [\,  e^{\eta \, U} \,] < \infty.$ Then there exist $0 < \chi_0 < \eta$ such that for all $0< \chi \le \chi_0$ there exists $0<\gamma<1, $ such that
	\begin{equation}
	\label{eq:Drift_Cond_Valid}
	{\mathbb E} \,[ \, e^{\chi (x+U )^+}\, ] \,\leq\, \gamma e^{\chi x} +1.  
	\end{equation}
\end{lemma}

\vspace*{-2mm}
A minor, but essential technical extension of the above arguments, in order to derive a uniform version of the drift condition, is stated in the lemma below:   
\begin{lemma}
	\label{lem:Uniform_Drift_Cond_Valid}
	Let $U(\theta) :=U_{\theta, 1}, \, \theta \in D,$ be a family of $\mathbb R$-valued random variables, satisfying Assumptions \ref{cond:Stability_of_Queue}, \ref{cond:Exponential_Moment_for_U_Uniform} and \ref{cond:Weak_Cont_of_U_theta}. Then there exist $0 < \chi_0 < \eta$ such that for all $0< \chi \le \chi_0$ there exists $0<\gamma<1, $ such that for all $\theta \in D$
	\begin{equation}
	\label{eq:Drift_Cond_Valid_K}
	{\mathbb E} \,[ \, e^{\chi (x+U(\theta) )^+}\, ]\, \leq\, \gamma e^{\chi x} +K.  
	\end{equation}
\end{lemma}
For the proofs of Lemmas \ref{lem:Drift_Cond_Valid} and
\ref{lem:Uniform_Drift_Cond_Valid} see the Appendix.

\begin{remark}
Taking into account the proof of Lemma \ref{lem:Drift_Cond_Valid} it is clear that all we need to show to prove Lemma \ref{lem:Uniform_Drift_Cond_Valid} is that a uniform version of \eqref{eq:Drift_Cond_Valid_Proof_2} is valid, i.e. that there exists a $0 < \chi_0 < \eta$ and some $\varepsilon > 0$ such that for all $0< \chi \le \chi_0$ and for all $\theta \in D$
\begin{equation}
\label{eq:Uniform_Drift_Cond_Valid_Proof_1}
{\frac {{\mathbb E} \,[ \, e^{\chi U(\theta) )}\, ]  - 1} {\chi}} \le - \varepsilon < 0.
\end{equation}
\end{remark}

Let us define the family of functions $g(\chi, \, \theta) : = {\mathbb E} \,[ \, e^{\chi U(\theta)}\, ].$ By Assumption \ref{cond:Exponential_Moment_for_U_Uniform} it is readily seen that the random variables $e^{\chi U(\theta)}$ are uniformly integrable for $\chi < \eta$, and, therefore, by Assumption \ref{cond:Weak_Cont_of_U_theta} it follows that $g(\chi, \, \theta)$ is continuous in $\theta.$ The desired claim \eqref{eq:Uniform_Drift_Cond_Valid_Proof_1} now follows from the following lemma formulated in the context of convex analysis:

\begin{lemma} 
	\label{lem:Convex_Uniform_Negative_Slope}
	Let  $g(\chi, \, \theta)$ be a family of convex functions in the variable $\chi$ 
	with $0 \le \chi < \eta$ and $\theta \in D \subset \mathbb R^k $ with $D$ being a compact set, such that 
	\begin{itemize}
		\item $g(0, \, \theta) =1$ for all $\theta \in D,$
		\item  for all fixed $\theta \in D$ we have 
		\begin{equation}
		\label{cond:Convex_Negative_Slope}
		\inf_{0 < \chi < \eta}{\frac { g(\chi, \, \theta)   - 1} {\chi}} < 0,
		\end{equation}
		\item for all fixed $0 \le \chi < \eta$ the function $g(\chi, \, . \,)$ is continuous in $\theta.$
	\end{itemize}
	Then there exists $\chi_0 >0$ such that for $0< \chi \le \chi_0$ we have 
	\begin{equation}
	\label{eq:Convex_Uniform_Negative_Slope}
	\sup_{\theta \in D} \, {\frac { g(\chi, \, \theta)   - 1} {\chi}} < 0.
	\end{equation}
	It follows that \, $\sup_{\theta \in D} \,  g(\chi, \, \theta)  < 1$ for $0< \chi \le \chi_0.$ 
\end{lemma}

\begin{remark} It also readily follows that 
\begin{equation*}
\label{eq:Convex_Uniform_Negative_Slope_Remark}
\sup_{\theta \in D} \, \inf_{0 < \chi < \eta}{\frac { g(\chi, \, \theta)   - 1} {\chi}}
= \inf_{0 < \chi < \eta}  \, \sup_{\theta \in D} {\frac { g(\chi, \, \theta)   - 1} {\chi}} < 0.
\end{equation*}
\end{remark}

{\bf Local Minorization for $P_{\theta}^r(x, \, .)$.} As for the local minorization condition it will be verified in a form slightly stronger than \tb{our Assumption \ref{cond:Local Minorization Extended to Pr Uniform}}, by showing that for any fixed $R >0$ there exists some integer $r \ge 1,$ which may depend on $R,$ such that 
\begin{equation}
\label{eq:Local Minorization for Pr for queues}
\inf_{\theta \in D} \, \inf_{0 \le x \le R} P_{\theta}^r (x, \{0\}) > 0.
\end{equation}
\tb{Indeed, the above inequality implies Assumption \ref{cond:Local Minorization Extended to Pr Uniform}. To see this, take an arbitrary $R_r$, as in Assumption \ref{cond:Local Minorization Extended to Pr Uniform}, satisfying $R_r>2K_r/(1-\gamma_r).$ Then the set 
\begin{equation}
{\cal C}_r=\{x\in{\mathbf X} :  V(x) \leq R_r \} = \{x\in{\mathbf X} :  e^{\chi x} \leq R_r \}
\end{equation}
is of the form $\{0 \le x \le R\}$ with some $R$. Letting $\bar{\mu}_r$ denote the probability measure assigning unit mass to $0$ inequality \eqref{eq:Local Minorization for Pr for queues} implies $\P_\theta^r(x,A)\geq\bar{\alpha}_r \bar{\mu}_r(A)$ with some $\bar{\alpha}_r\in(0,1)$ for all $\theta \in \Theta$ and $x \in {\cal C}_r,$ as postulated by \eqref{eq:Local Minorization Extended to Pr Uniform}. 
}

To prove  \eqref{eq:Local Minorization for Pr for queues} we will use arguments familiar in queuing theory, but to establish uniform bounds extra care is needed. Consider first a fixed $\theta \in D,$ and let $R>0$ be any fixed real number, defining the bounded 
$[0,R].$ Let $0 \le x \le R$ and consider the $r$-step transition probability 
$P_{\theta}^r(x, \{0\}) = P(W_{\theta, r}(x) = 0)$ for some integer $r\geq 1$.

\begin{lemma}\label{minori}
	There is $\epsilon>0$ such that 
	\begin{equation}
	\upsilon:=\inf_{\theta\in D}P(U_{\theta, 1}< -\epsilon)>0.
	\end{equation}
\end{lemma}

\begin{corollary}\label{corollary:3} For each $R>0$ 
	there is $r\in\mathbb{N}$ such that 
	\begin{equation}
	\inf_{\theta \in D}\inf_{0\leq x\leq R}P(W_{\theta, r}(x)=0)>0.
    \vspace{2mm}
    \end{equation}
\end{corollary}

\vspace{-2mm}
A nice additional result is the following: let $W_{\theta, s}^\ast$ denote the stationary solution of the queue dynamics given by 
\begin{equation}
W_{\theta, s}^\ast := \, \max_{-\infty \le k \le s} ( U_{\theta, k} + \ldots + U_{\theta, s})^+.
\end{equation}
Then 
\begin{equation}
\inf_{\theta \in D}\inf_{0\leq x\leq R}P(W_{\theta, r}(x)=0) \ge C_r \inf_{\theta \in D}\, P (W_{\theta, r}^\ast = 0), 
\end{equation}
where $C_r\to 1$ exponentially fast as $r\to \infty$. Moreover,
\vspace*{-1mm}
\begin{equation}
\inf_{\theta \in D}\, P (W_{\theta,r}^\ast = 0) > 0.
\end{equation}

{\bf Lipschitz Continuity of $P_\theta$.} In order to verify Assumption \ref{cond:Lipschitz of P theta Dirac x by Miklos} we will need to strengthen our assumptions on $S_{\theta, 0}$ and $T_{\theta, 1}$. We may consider various scenarios, briefly discussed below, both of them implying a common condition on $U_{\theta, 1}$ as follows: 

\begin{assumption}
	\label{cond:Lip_on_U_theta}
	The probability distribution function of $U_{\theta, 1}$ has a density function for all $\theta\in D,$ denoted by $\zeta_{\theta}(\cdot)$ 
	There exist $\eta^{\prime \prime} > 0$ and $C^{\prime \prime}>0$ such that for all $\theta, \, \theta^{\prime} \in \Theta$ and all $x\in\mathbb{R}$, and it holds that 
	\begin{equation}
	\vert \zeta_{\theta}(x) - \zeta_{\theta^{\prime}}(x) \vert \,\leq\, C^{\prime \prime} \, e^{-\eta^{\prime \prime} \vert x \vert } \vert \, \theta - \theta^{\prime} \vert.
	\end{equation}
\end{assumption}

Assumption \ref{cond:Lip_on_U_theta} can be conveniently verified by imposing the following assumption on $S_{\theta, 0}$ and $T_{\theta, 1}: =T_{1}$, when $T_{\theta, 1}$ is assumed to be independent of $\theta:$ the probability distribution functions of $S_{\theta, 0}$ has a density function for all $\theta\in D,$ denoted by $\xi_{\theta}(\cdot),$ and there exist $C^{\prime},\eta^{\prime}>0$ such that for all $\theta, \theta^{\prime} \in \Theta$ and all $x\in\mathbb{R}_+$, it holds that $\vert \xi_{\theta}(x) - \xi_{\theta^{\prime}}(x) \vert \leq C^{\prime} e^{-\eta^{\prime} x} \vert {\theta} - {\theta^{\prime}} \vert.$ Moreover the probability distribution function of $T_{1}$ has a density function denoted by $\kappa(\cdot)$ such that for all $x\in\mathbb{R}_+$ it holds that $\kappa(x) \leq C^{\prime} e^{-\eta^{\prime} x}. $

The first part of the above auxiliary assumption can be conveniently 
checked by requiring the existence of a density function $\xi_{\theta}(\cdot)$ such that the mapping $(\theta,x) \to \xi_{\theta}(x)\in \mathbb{R}_{+}$ is measurable,  and for each fixed $x$ continuously differentiable in $\theta\in\Theta$, moreover 
there exist $C^{\prime},\eta^{\prime}>0$ such that for all $\theta\in \Theta$ and all $x\in\mathbb{R}_+$
it holds that $\Vert {\frac {\partial} {\partial {\theta}}} \xi_{\theta}(x) \Vert \leq C^{\prime} e^{-\eta^{\prime} x}.$ The proof is  
readily obtained by the mean-value theorem. Incidentally, it also follows that the 
law of $S_{\theta, 0}$ is continuous in total variation, and, a fortiori, also weakly, 
implying Assumption  \ref{cond:Weak_Cont_of_U_theta}.

A condition reciprocal to the above is obtained by interchanging the role of $S_{0}$ and $T_{1},$ i.e. assuming that $S_{0}$ does not depend on $\theta,$ while $T_{1}=T_{\theta, 1}$ does. We note that in this case the requirement that the density function of $S_{0},$ denoted by $\xi(\cdot),$ satisfies $\xi(x) \leq C^{\prime}  e^{-\eta^{\prime} x} $ for all $x\in\mathbb{R}_+$ implies Assumption \ref{cond:Exponential_Moment_for_U_Uniform} with any $\eta < \eta^{\prime}$.

In order to verify Assumption \ref{cond:Lipschitz of P theta Dirac x by Miklos}, let us consider the transition probabilities $P_{\theta}(x,A)$. 
For any Borel set $A\subset\mathbb{R}_{+}, \, 0 \notin A$ we can write
\begin{align}
P_{\theta}(x,A)&=P(W_{\theta, 1}(x)\in A)= \mathbb E \, [\mathbf{1}_A(x+U_{\theta, 1})] \nonumber \\
& =\int_{-\infty}^{\infty}  \mathbf{1}_A(x+z) \, \zeta_\theta(z) \, \dd z. 
\end{align}
\tb{State} $0,$ being an atom for $P_{\theta}(x,\,.),$ is reached with probability\hspace*{-3mm}
\begin{align}
P_{\theta}(x, \{0\}) &= P(U_{\theta, 1} \leq -x ) = 
\int_{-\infty}^{-x}  \zeta_\theta(z) \, \dd z =  \nonumber \\ &\int_{-\infty}^{\infty}  \mathbf{1}_{(-\infty, 0]}\, (x+z) \, \zeta_\theta(z) \, \dd z.  
\end{align}

\begin{lemma}
	\label{lem:Lip_Stated}
	Let Assumption \ref{cond:Lip_on_U_theta} hold, and let $V(x):=e^{\chi x}$, $x\in\mathbb{R}_{+},$ with $\chi < \eta^{\prime \prime}.$ Then for all $\phi \in {\cal L}_V$ with $||\phi||_{\beta}\leq 1$ and all $\theta,\theta'\in\Theta$ we have with some $L>0$ 
	$$
	\int_{\mathbb{R}_{+}}\phi(y) (P_{\theta}(x,\dd y)-P_{\theta'}(x,\dd y))\,\leq\, L (1+\beta V(x))|\theta-\theta'|.
	$$	
\end{lemma}

\section{Discussion}		
\label{sec:Discussion}
\medskip
\noindent
We have have re-visited a key technical issue in the theory of stochastic approximation in a Markovian framework, developed in \cite{benveniste1990}: 
the Lipschitz-continuity of the solution of the associated Poisson equation w.r.t. the parameter $\theta,$ characterizing the system dynamics. 
A set of simple conditions have been formulated under which the desired Lipschitz-continuity can be established, significantly simplifying the relevant conditions and results of \cite{benveniste1990}. 
We demonstrated the utility of a  
powerful, off-beat result on the stability of Markov chains in \cite{hairer2011yet}, proving that the transition kernels are contractions in the space of differences of probability measures in a suitable metric. 

The uniform drift condition, Assumption \ref{cond:Uniform Drift Condition}, and its relaxation, restating it for some power of the kernel, $\P^r_\theta(x,A),$ see Assumption \ref{cond:Uniform Drift Condition for Pr theta},  
is akin to condition (A'.5)(i), p.\ 290 of \cite{benveniste1990}, but allowing more general Lyapunov-functions $V(\cdot)$. As for the condition on the Lipschitz-continuity of the kernel our Assumption \ref{cond:Lipschitz of P theta Dirac x by Miklos} is in most aspects significantly less restrictive than the corresponding conditions of a key result, {Theorem 6, p.\, 262} of \cite{benveniste1990}. Noting {some} overlaps between the proofs of the latter result of \cite{benveniste1990} and that of our Theorem \ref{thm:Poisson Lipschitz}, a question for future research is whether our Assumption \ref{cond:Lipschitz of P theta Dirac x by Miklos} can be relaxed by using a smaller class of test functions.
To complete the loop, the technology presented here is going to be applied for the ODE analysis of recursive estimators along the lines of \cite{benveniste1990}.

The viability of our results has been demonstrated on the modification of a textbook example of a queuing system, allowing open-loop control. The extension of our analysis to complex networks, 
with several servers and/or customers allowing feedback control is an attractive and challenging problem. 

The setup of our paper is suitable for the analysis of certain cyber-physical systems, incorporating systems described by stochastic partial differential equations (SPDEs). A prototype for Markov processes arising in such context is given in the lecture notes 
\cite{berglund2019introduction}.
Finally, recent intense interest in online machine learning will definitely inspire further applications, especially in various stochastic gradient methods and reinforcement learning where the Markovian setup is the standard choice.

\backmatter

\bigskip
\bmhead{Acknowledgements}

{A. Car\`e and B. Cs. Cs\'aji were (partially) supported by the European Commission through the H2020 project Centre of Excellence in Production Informatics and Control (EPIC, 739592). 
B. Cs. Cs\'aji and L.\ Gerencs\'er were supported by the European Union within the framework of the National Laboratory for Autonomous Systems (RRF-2.3.1-21-2022-00002). M.~R\'asonyi and B.~Gerencs\'er were supported by NRDI (National Research, Development and Innovation Office) grant KKP 137490. B.~Gerencs\'er was also supported by the János Bolyai Research Scholarship of the Hungarian Academy of Sciences. M.~R\'asonyi was also supported by the NRDI grant K 143529. The authors thank to M\'at\'e Gerencs\'er for his comments on the potential of Markov processes defined by SPDEs.}\\

\begin{appendices}
\section{Proofs}\label{secA1}

\begin{proof}[Proof of Corollary \ref{cor:sigma beta eq rho beta for pairs}]
	Indeed, since $\{\varphi: \|\varphi\|_\beta \le 1\} \subseteq \{\varphi: \vertiii{\varphi}_\beta \le 1\}$, we 
    get $\rho_\beta(\mu_1,\mu_2) \le \sigma_\beta(\mu_1,\mu_2).$ On the other hand, take $\varepsilon >0$ and let $\varphi$ be such that $\vertiii{\varphi}_\beta \le 1$ and
	\begin{equation}
	\int_{\mathbf X} \varphi(x) (\mu_1-\mu_2)(\mathrm{d}x) \ge \sigma_\beta(\mu_1,\mu_2) - \varepsilon.     
	\end{equation}
	By Definition \ref{def:Triple Norm} there exists a constant $c$ such that $\vertiii{\varphi}_\beta=\|\varphi+c\|_\beta$. Thus, $\|\varphi+c\|_\beta \leq 1$, therefore   
\begin{align}
\rho_\beta(\mu_1,\mu_2)	&\ge \int_{\mathbf X} (\varphi(x) + c) (\mu_1-\mu_2)(\mathrm{d}x)\nonumber\\
& =  \int_{\mathbf X} \varphi(x) (\mu_1-\mu_2)(\mathrm{d}x)\nonumber\\
& \ge (\sigma_\beta(\mu_1,\mu_2) - \varepsilon).
\end{align}
	Since $\varepsilon$ is arbitrary, we get that $\rho_\beta(\mu_1,\mu_2) \ge \sigma_\beta(\mu_1,\mu_2).$ Combining with the opposite inequality, we get the claim.
\end{proof}

\begin{proof}[Proof of Lemma \ref{lemma:P theta Lipschitz with mu}]
For the integral of the left hand side of \eqref{eq:Miklos Lipschitz Dirac x Applied to varphi} we apply Fubini's theorem to get
\begin{equation}
\label{eq:Fubini for mu P phi}
\int_{\mathbf X} \left( \int_{\mathbf X}  \varphi(y) \P_{\theta} (x, \dd y)\right) \mu(\mathrm{d}x) =
\int_{\mathbf X} \varphi(y) \eta(\dd y),
\end{equation}
where the measure $\eta=P_{\theta}\mu$ is defined as usual by $\eta(A) = \int_{\mathbf X} \P_{\theta} (x, A)  \mu(\mathrm{d}x)$. The measure $\eta$ is finite, since $\mu(\mathbf X) < \infty$.
The application of Fubini's theorem is justified since 
\begin{align*}
\iint | \varphi(y) | \P_{\theta} (x, \dd y) \mu(\mathrm{d}x) & \leq \iint (1 + \beta V(y)) \P_{\theta} (x, \dd y) \mu(\mathrm{d}x)\\
& \leq \int (1 + \beta (\gamma V(x) + K)) \mu(\mathrm{d}x),
\end{align*}
and the right hand side is finite. Using the same argument for $\theta'$, altogether we obtain for the integral of \eqref{eq:Miklos Lipschitz Dirac x Applied to varphi} 
\begin{equation*}
\int_{\mathbf X} \varphi(y)\left( \P_{\theta}\mu(\dd y) - \P_{\theta^{\prime}}\mu(\dd y)\right) \leq  L_P|\theta-\theta^{\prime}|\mu(1+\beta V).
\end{equation*}
Since $\varphi$ is arbitrary subject to $\|\varphi\|_{\beta} \leq 1$, we conclude that $\rho_\beta(\P_{\theta}\mu,\P_{\theta^{\prime}}\mu)= \sigma_\beta(\P_{\theta}\mu,\P_{\theta^{\prime}}\mu)$ is bounded by the right hand side of \eqref{eq:Miklos Lipschitz Dirac x Applied to varphi}, and we get the statement of the Lemma.
\end{proof}

\begin{proof}[Proof of Lemma \ref{lemma:P theta Lipschitz wrt signed eta}]
Consider the Hahn-Jordan decomposition $\eta = \eta^+ - \eta^-$, as recalled after the lemma itself. Then
\begin{equation*}
\sigma_\beta((\P_{\theta}- \P_{\theta^{\prime}})\eta) \leq 
\sigma_\beta((\P_{\theta}- \P_{\theta^{\prime}})\eta^+) + \sigma_\beta((\P_{\theta}- \P_{\theta^{\prime}})\eta^-).
\end{equation*}
Using Lemma \ref{lemma:P theta Lipschitz with mu} for both terms we get the   upper bound:
\begin{equation}
L_P|\theta-\theta^{\prime}|\eta^+(1+\beta V)+ L_P|\theta-\theta^{\prime}|\eta^-(1+\beta V).
\end{equation}
Noting that $\eta^+ +\eta^- = |\eta|$ the lemma follows.

\end{proof}

\begin{proof}[Proof of Lemma \ref{lemma:Lipschitz for Pn with eta}]
We can estimate $\sigma_\beta(\P^n_{\theta}\eta,\P^n_{\theta^{\prime}}\eta)$ from above, using a kind of telescopic sequence of triangular inequalities, leading to the upper bound 
\begin{align}
 \label{eq:P1n vs P2n Triangle}
&\sum_{k=0}^{n-1}\sigma_\beta( \P^{n-k}_{\theta}\P^{k}_{\theta^{\prime}} \eta, \P^{n-k-1}_{\theta}\P^{k+1}_{\theta^{\prime}} \eta  ) \nonumber \\
= &\sum_{k=0}^{n-1}\sigma_\beta( \P^{n-k-1}_{\theta}\P_{\theta}\P^{k}_{\theta^{\prime}} \eta, \P^{n-k-1}_{\theta}\P_{\theta^{\prime}}\P^{k}_{\theta^{\prime}} \eta ).
\end{align}
Note that the measures $\P_{\theta}\P^{k}_{\theta^{\prime}} \eta$ and $\P_{\theta^{\prime}}\P^{k}_{\theta^{\prime}} \eta$ satisfy $\P_{\theta}\P^{k}_{\theta^{\prime}} \eta(\mathbf X) =\P_{\theta^{\prime}}\P^{k}_{\theta^{\prime}} \eta(\mathbf X)$, hence their $\sigma_\beta(.,.)$ distance is well-defined, see Definition \ref{def:sigma beta for pairs}.

Using the contraction property of the kernels $\P^{n-k-1}_{\theta}$, see Proposition \ref{prop:Contraction of sigma beta}, we obtain the upper bound
\begin{equation}
\label{eq:P1n vs P2n Triangle plus Contraction}
\begin{aligned}
\sum_{k=0}^{n-1} \alpha^{n-k-1}\sigma_\beta( \P_{\theta}\P^{k}_{\theta^{\prime}} \eta, \P_{\theta^{\prime}}\P^{k}_{\theta^{\prime}}\eta  ).
\end{aligned}
\end{equation}
For the $k$-th term we estimate $\sigma_\beta( \P_{\theta}\P^{k}_{\theta^{\prime}} \eta - \P_{\theta^{\prime}}\P^{k}_{\theta^{\prime}}\eta  )$ from above applying Lemma \ref{lemma:P theta Lipschitz wrt signed eta} with $\P^{k}_{\theta^{\prime}} \eta$ taking the role of $\eta$ to get the following upper bound for 
\eqref{eq:P1n vs P2n Triangle plus Contraction}:
\begin{equation}
\label{eq:Sum of Pk theta Lipschitz}
L_P |\theta-\theta^{\prime}| \sum_{k=0}^{n-1}  \alpha^{n-k-1}
|\P_{\theta^{\prime}}^k\eta|(1+\beta V).
\end{equation}

Note that by the consequence of the drift condition given in inequality \eqref{eq:Uniform Drift Condition w eta} we can bound $|\P_{\theta}^k\eta|(V)$ for a general $\theta$ by
\begin{equation}
|\P_\theta^k\eta|(V) \leq \gamma |\P_\theta^{k-1}\eta|(V) + K |\P_\theta^{k-1}\eta|(\mathbf X).
\end{equation}
Noting that $ |\P_\theta^{k-1}\eta|(\mathbf X) \leq |\eta|(\mathbf X),$ and iterating the above inequality, we get
\begin{align}
|\P_\theta^k\eta|(V) &\leq \gamma^2 |\P_\theta^{k-2}\eta|(V) + \gamma K |\eta|(\mathbf X) + K |\eta|(\mathbf X) \nonumber \\[1mm]
&\cdots \nonumber \\
& \leq  \gamma^k |\eta|(V) + \sum_{\ell=0}^{k-1}\gamma^\ell K |\eta|(\mathbf X) \nonumber \\
&\leq  \gamma^k |\eta|(V)+\frac{K}{1-\gamma} |\eta|(\mathbf X).
\label{eq:Pk mu V upper bound}
\end{align}
By plugging \eqref{eq:Pk mu V upper bound} into the sum in \eqref{eq:Sum of Pk theta Lipschitz}, we get the upper bound
\begin{equation*}
\sum_{k=0}^{n-1} \alpha^{n-k-1} \left( |\eta|(\mathbf X)+\beta \left(\gamma^k |\eta|(V)+\frac{K}{1-\gamma}|\eta|(\mathbf X)\right)\right).
\end{equation*}
We can write the latter expression as
\begin{align}
\!\!\!\beta \alpha^{n-1}\sum_{k=0}^{n-1} \left(\frac{\gamma}{\alpha}\right)^k\! |\eta|(V)+\!\left(\!1 + \beta \frac{K}{1-\gamma}\right)\! \sum_{k=0}^{n-1}\alpha^k |\eta|(\mathbf X). \label{eq:Sum of Pk theta Lipschitz UB V2}
\end{align}
Summarizing the inequalities \eqref{eq:P1n vs P2n Triangle} to
\eqref{eq:Sum of Pk theta Lipschitz UB V2}, taking into account $\alpha>\gamma$ (see Remark \ref{rem:alpha Majorates gamma}), and bounding  the geometric sums in \eqref{eq:Sum of Pk theta Lipschitz UB V2} with their limit values 
we get the upper bound
\begin{equation}
\label{eq:Lipschitz for Pn fnl}
\frac{1}{\-1-\alpha}\left(1 + \beta \frac{K}{1-\gamma}\right) \vee \frac{\alpha^{n}}{\alpha-\gamma},
\end{equation}
from which the claim follows by setting $n=0$.
\end{proof}

\begin{proof}[Proof of Corollary \ref{corr:Lipschitz for mu stars}]
Note that for any initial probability measure $\mu \in {\cal M}_V$, we have by the triangle inequality
\begin{equation*}
\sigma_\beta(\mu_{\theta}^\ast,\mu_{\theta^{\prime}}^\ast) \leq \sigma_\beta(\mu_{\theta}^\ast,\P^n_{\theta}\mu) +
\sigma_\beta(\P^n_{\theta}\mu,\P^n_{\theta^{\prime}}\mu) + \sigma_\beta(\P^n_{\theta^{\prime}}\mu,\mu_{\theta^{\prime}}^\ast).
\end{equation*}
Letting $n \rightarrow \infty$ the first and the last terms on the r.h.s. converge to zero by Proposition \ref{prop:Contraction of sigma beta}. 
Taking $\mu= \delta_x$, the middle term is upper bounded, for any $n$, in view of Lemma \ref{lemma:Lipschitz for Pn with eta} by 
\begin{equation}
L_P |\theta-\theta^{\prime}|  \cdot C_P ( 1 + \beta V(x)).
\end{equation}
Note that $C_P$ can be replaced by what is given in \eqref{eq:Lipschitz for Pn fnl}. Recalling that $\inf_x V(x) = 0$ by Remark \ref{rem:Constant Shifts of V}, and letting $n \to \infty$,  the claim follows with the constant $C_P^\prime$ stated in the corollary.
\end{proof}

\begin{proof}[Proof of Lemma \ref{lemma:Lipschitz for Pn etaX 0}] 
The starting point of the proof is the inequality, obtained by combining 
\eqref{eq:P1n vs P2n Triangle} -- \eqref{eq:P1n vs P2n Triangle plus Contraction}, applicable also for signed meausures such that $|\eta|(1+\beta V) < \infty$: 
\begin{equation}
\label{eq:Lipschitz for Pn etaX 0 eq1}
\sigma_\beta(\P^n_{\theta}\eta,\P^n_{\theta^{\prime}}\eta)  \leq \sum_{k=0}^{n-1} \alpha^{n-k-1}\sigma_\beta( \P_{\theta}\P^{k}_{\theta^{\prime}} \eta, \P^{k+1}_{\theta^{\prime}} \eta  ).
\end{equation}
A key point is the observation that since  $\eta(\mathbf X) = 0$, $\P^{k}_{\theta^{\prime}} \eta$ converges exponentially fast to the zero measure, see Proposition \ref{prop:Contraction of sigma beta}. To estimate the $k$\,th term of \eqref{eq:Lipschitz for Pn etaX 0 eq1}, we apply Lemma \ref{lemma:P theta Lipschitz wrt signed eta} and Proposition \ref{prop:sigma beta eq rho beta}, \eqref {eq:sigma beta eq rho beta},
\begin{align}
\sigma_\beta( (\P_{\theta}-\P_{\theta^{\prime}})\P^{k}_{\theta^{\prime}}\eta) &\leq L_P|\theta-\theta^{\prime}| |\P^{k}_{\theta^{\prime}}\eta|(1+\beta V)\nonumber\\
& =  L_P|\theta-\theta^{\prime}| \sigma_{\beta}(\P^{k}_{\theta^{\prime}}\eta).
\end{align}
Now applying Proposition \ref{prop:Contraction of sigma beta} and Proposition \ref{prop:sigma beta eq rho beta}, \eqref{eq:sigma beta eq rho beta}, again, we get the upper bound:
\begin{equation}
L_P|\theta-\theta^{\prime}| \alpha^k \sigma_\beta(\eta) = L_P|\theta-\theta^{\prime}| \alpha^k |\eta|(1+ \beta V).
\end{equation}
Inserting this into \eqref{eq:Lipschitz for Pn etaX 0 eq1}, we get the desired upper bound.
\end{proof}

\begin{proof}[Proof of Lemma \ref{lem:One Boundedness of P in 2 and 3 Norm}]
	    To simplify the notations we write $\P_\theta = \P.$ We have $|\varphi(x)| \le \Vert \varphi \Vert_\beta (1 + \beta V(x))$ from which we get 
		\begin{align}
		\hspace*{-2mm}|\Ps \varphi(x)|  \le \Ps |\varphi |(x) &\le \Vert \varphi \Vert_\beta (1 + \Ps \beta V(x)) \nonumber \\ 
		& \le \Vert \varphi \Vert_\beta (1 + \beta (\gamma_1 V(x) + K_1)),
		\end{align}
		by Assumption \ref{cond:Uniform One Step Growth Condition for P theta}. The last term on the right hand side is majorized by $ \alpha^{\prime} (1 + \beta V(x))$ with $\alpha^{\prime} = \gamma_1 \vee (1 + \beta K_1),$ proving the first half of \eqref{eq:One Step Stability of P star in 2 and 3 Norm}.
		To prove the second half of \eqref{eq:One Step Stability of P star in 2 and 3 Norm} recall that for any $\psi \in {\cal L}_V$ we have $\vertiii{\psi}_\beta = \min_c \| \psi + c \|_\beta.$ Hence for any constant $c$ we have 
		\begin{equation*}
        \vertiii {\Ps \varphi}_\beta  = \vertiii {\Ps \varphi + c}_\beta \le \Vert {\Ps \varphi + c}  \Vert_\beta = \Vert {\Ps (\varphi + c)}  \Vert_\beta. 
        \label{eq:Pphi 3 norm vs Pphi 2 norm}
        \end{equation*}
		Apply the first inequality of \eqref{eq:One Step Stability of P star in 2 and 3 Norm} with 
		$\varphi + c$ replacing $\varphi:$  
		\begin{equation} 
		\Vert \Ps (\varphi + c) \Vert_{\beta} \le \alpha^{\prime} \Vert \varphi + c \Vert_\beta. 
		\end{equation}
		Choosing $c$ so that  $\Vert {\varphi + c} \Vert_\beta = \vertiii{\varphi}_\beta$ yields the claim. 
\end{proof}

			\begin{proof}[Proof of Lemma \ref{lemma:Lipschitz for Pn with OneStepGrowth}] The proof is obtained by a simple modification of the proof of Lemma \ref{lemma:Lipschitz for Pn with eta}. Estimate $\sigma_\beta(\P^n_{\theta}\eta,\P^n_{\theta^{\prime}}\eta)$ using a sequence of triangular inequalities to get 
			    \vspace{-1mm}
				\begin{equation*}
				\label{eq:P1n vs P2n Triangle CP}
				\sigma_\beta(\P^n_{\theta}\eta,\P^n_{\theta^{\prime}}\eta) 
				\,\leq\,  \sum_{k=0}^{n-1}\sigma_\beta( \P^{n-k}_{\theta}\P^{k}_{\theta^{\prime}} \eta, \P^{n-k-1}_{\theta}\P^{k+1}_{\theta^{\prime}} \eta  ).
                \end{equation*}
Consider the $k$th term and apply Lemma \ref{lemma:one step Stab for sigma beta}, or  \eqref{eq:one step Stab for sigma beta},  repeatedly $n-k-1$ times setting $\eta_1 = \P_{\theta}\P^{k}_{\theta^{\prime}} \eta$ and $\eta_2 =  \P^{k+1}_{\theta^{\prime}} \eta$:
				\begin{equation}
				\sigma_\beta(\P^{n-k-1}_{\theta} \eta_1, \P^{n-k-1}_{\theta} \eta_2) \leq (\alpha^{\prime})^{n-k-1} \sigma_\beta(\eta_1,\eta_2).
				\end{equation}
Note that the conditions of Lemma \ref{lemma:one step Stab for sigma beta} are satisfied for $\eta_1,\eta_2$: obviously $\eta_1(\mathbf X) = \eta_2(\mathbf X)=\eta(\mathbf X) < \infty$ and $|\eta_i|(V) < \infty$, for $i=1,2$ due to the repeated application of the one-step growth condition. Combining the last two inequalities we get:
				\begin{equation}
				\label{eq:P1n vs P2n Triangle plus Contraction Modified}
				\sigma_\beta(\P^n_{\theta}\eta,\P^n_{\theta^{\prime}}\eta) \le \sum_{k=0}^{n-1} (\alpha^{\prime})^{n-k-1} \sigma_\beta( \P_{\theta}\P^{k}_{\theta^{\prime}} \eta, \P^{k+1}_{\theta^{\prime}} \eta  ) .
				\end{equation}
				Consider the $k$-th term, and recall the Lipschitz continuity of $(\P_{\theta}),$ Assumption \ref{cond:Lipschitz of P theta Dirac x by Miklos}, implying Lemma \ref{lemma:P theta Lipschitz wrt signed eta}. Applying the latter for the signed measure $\P^{k}_{\theta^{\prime}} \eta$ we get the upper bound 
				\begin{equation}
				\label{eq:Sum of Pk theta Lipschitz CP}
				\sum_{k=0}^{n-1}   (\alpha^{\prime})^{n - k-1}
				L_P |\theta-\theta^{\prime}|\cdot |\P_{\theta^{\prime}}^k\eta|(1+\beta V).
				\end{equation}
	
To estimate $|\P_{\theta^{\prime}}^k\eta|(V)$, we use \eqref{eq:Pk mu V upper bound under one step growth}, restated as  
  \begin{align}
  \label{eq:Pk mu V upper bound under one step growth CP}
    |\P_{\theta'}^k\eta|(V) \le \gamma_1^k \left(|\eta|(V)+K_1|\eta|(\mathbf X)/(\gamma_1-1) \right).
  \end{align}
By plugging 
this into \eqref{eq:Sum of Pk theta Lipschitz CP}, we get the upper bound
				\begin{equation*}
				L_P |\theta-\theta^{\prime}|  \sum_{k=0}^{n-1}  (\alpha^{\prime})^{n-k-1} 
				\left( 1+\beta \gamma_1^k \left(|\eta|(V)+\frac{ K_1 |\eta|(\mathbf X)}{\gamma_1-1} \right)\right).
				\end{equation*}
				The first term in the above sum is bounded from above by 			$(\alpha^{\prime})^{n}/ (\alpha^{\prime} -1).$ The second term can be written as  
					\begin{equation}
				\sum_{k=0}^{n-1}  (\alpha^{\prime})^{n-k-1} \gamma_1^k
				\beta \left( |\eta|(V)+\frac{ K_1 |\eta|(\mathbf X)}{\gamma_1-1}\right).
				\end{equation}
Recall that $\gamma_1 \le \alpha^{\prime}$, hence a simplified upper bound is
\begin{equation}
				n (\alpha^{\prime})^{n-1} 
				\beta \left( |\eta|(V)+\frac{ K_1 |\eta|(\mathbf X)}{\gamma_1-1}\right),
				\end{equation}
and $n (\alpha^{\prime})^{n-1}$ can be bounded from above by $C (\alpha^{\prime\prime})^{n-1}$ for any $\alpha^{\prime\prime} > \alpha^{\prime},$ where $C$ depends only on $\alpha^{\prime}$, and $\alpha^{\prime\prime}.$
Summarizing the inequalities \eqref{eq:P1n vs P2n Triangle plus Contraction Modified} to \eqref{eq:Pk mu V upper bound under one step growth CP}, and 
				the arguments that follow, we get the claim.
			\end{proof}

\begin{proof}[Proof of Lemma \ref{lem:Drift_Cond_Valid}]
	Since the function $g(\chi) : = {\mathbb E} \,[ \, e^{\chi U}\, ]$ is convex in $\chi,$ the finite difference quotients are monotone non-increasing for $\chi \downarrow 0$ with negative limit:   
	\begin{equation}
	\label{eq:Drift_Cond_Valid_Proof_1}
	\lim_{\chi \downarrow 0}{\frac {{\mathbb E} \,[ \, e^{\chi U }\, ]  - 1} {\chi}} \downarrow {\mathbb E} \, U <0.
	\end{equation}
	Hence there exists a $0 < \chi_0 < \eta$ and some $\varepsilon > 0$ such that for all $0< \chi \le \chi_0$ 
	\begin{equation}
	\label{eq:Drift_Cond_Valid_Proof_2}
	\sup_{0 < \chi \le \chi_0}{\frac {{\mathbb E} \,[ \, e^{\chi U }\, ]  - 1} {\chi}} \le - \varepsilon < 0.
	\end{equation}
	It follows that ${\mathbb E} \,[ \, e^{\chi U }\, ] \le  1  - {\chi}  \varepsilon, $ and thus we get 
	\begin{align}
	\label{eq:Drift_Cond_Valid_Proof_3}
	&{\mathbb E} \,[ \, e^{\chi (x + U)^+}\, ] \nonumber \\
	= \, &{\mathbb E} \,[ \mathbf{1}_{\{x + U \ge 0\}}\, e^{\chi (x + U)^+}\, ] + {\mathbb E} \,[ \mathbf{1}_{\{x + U < 0\}}\, e^{\chi (x + U)^+}\, ]\nonumber
	 \\   \le \, &{\mathbb E} \,[ \, e^{\chi (x + U)}\, ] + 1 \le 
	e^{\chi x} (1  - {\chi}\varepsilon)  + 1.
	\end{align}
\end{proof}
\begin{proof}[Proof of Lemma \ref{lem:Convex_Uniform_Negative_Slope}]
	Let us define the function
	\begin{equation}
	\label{eq:Convex_Uniform_Negative_Slope_Proof_1}
	g(\chi):= \sup_{\theta \in D} g(\chi, \, \theta),
	\end{equation}
	\tb{for $0 \le \chi < \eta.$} Obviously, $g(\cdot)$ is convex and $g(0)=1.$ The claim of the lemma can be then restated as saying that there exists $\chi_0 >0$ such that for $0< \chi \le \chi_0$ we have 
	\begin{equation}
	\label{eq:Convex_Uniform_Negative_Slope_with_sup}
	{\frac { g(\chi)   - 1} {\chi}} < 0 \qquad {\rm or} \qquad  \inf_{0 < \chi <  \chi_0}{\frac { g(\chi)   - 1} {\chi}} < 0.
	\end{equation}

	Assume that the claim is not true, and let $\chi_n \downarrow 0$ be a monotone sequence such that we have 
	\begin{equation}
	\label{eq:Convex_Uniform_Negative_Slope_with_sup_2}
	{\frac { g(\chi_n)   - 1} {\chi_n}} \ge 0 \qquad {\rm or} \qquad  { g(\chi_n)   \ge 1}.
	\end{equation}
	Let $\theta_n \in D$ be such that 
	\begin{equation*}
	\label{eq:Convex_Uniform_Negative_Slope_with_sup_3}
	g(\chi_n) = \sup_{\theta \in D} \, g(\chi_n, \, \theta) = \max_{\theta \in D} \, g(\chi_n, \, \theta) = g(\chi_n, \, \theta_n).
	\end{equation*}
	Due to the compactness of $D$ we can assume that $\theta_n \in D$ has a limit in $D,$ say $\lim \theta_n = \theta^\ast \in D.$ Consider now the function $g(\cdot \,, \theta^\ast)$ and choose a $\chi_0$ such that 
	\begin{equation}
	\label{eq:Convex_Uniform_Negative_Slope_with_sup_4}
	{\frac { g(\chi_0, \theta^\ast)   - 1} {\chi_0}} < 0 \qquad {\rm or} \qquad  g(\chi_0, \theta^\ast) =: 1- c < 1.
	\end{equation}
	The continuity of $g(\chi_0, . \,)$ in $\theta$ implies $g(\chi_0, \theta_n) \le 1- c/2  < 1$ for sufficiently large $n$. On the other hand the convexity of the function $g(.\,, \theta_n),$ and $g(0\,, \theta_n) = 1$ and $ g(\chi_n\,, \theta_n) \ge 1$ imply that for $\chi_0 > \chi_n$ we have $g(\chi_0, \theta_n) \ge 1,$ a contradiction, proving the claim. 
\end{proof}

\begin{proof}[Proof of Lemma \ref{minori}]
	By Lemma \ref{lem:Convex_Uniform_Negative_Slope}, it follows that for sufficiently small $\chi$ we 
	have $A:=\sup_{\theta \in D} \mathbb E [e^{\chi U_{\theta, 1}}] <1$ hence
\begin{align}
	&\sup_{\theta \in D} P(U_{\theta, 1}\geq -\epsilon) = \sup_{\theta \in D} P(e^{\chi U_{\theta, 1}}   
	\geq e^{-\chi\epsilon}) \leq \nonumber \\
	&\sup_{\theta \in D} \mathbb E [e^{\chi U_{\theta, 1}}]e^{\chi\epsilon}\leq Ae^{\chi\epsilon},
\end{align}
	which is strictly smaller than $1$ for $\epsilon$ small enough. Thus indeed, $v:= \inf_{\theta \in D} P(U_{\theta, 1} < -\epsilon) >0,$ as stated.
\end{proof}

\begin{proof}[Proof of Corollary \ref{corollary:3}]
	Indeed, let $\epsilon,\upsilon$ be as in Lemma \ref{minori} and choose $r$ so large that $r\epsilon>R$. Then, for all $\theta\in D$ and $0\leq x\leq R$,
	$P(W_{\theta, r}(x)=0)$ is bounded from below by 
	\begin{align}
	 &P  (W_{\theta, k}(x)\leq (x-k\epsilon)^+,\ \forall \, k=1,\ldots, r ) 	\ge \nonumber\\
 &P (U_{\theta,k}< -\epsilon,\ \forall \, k=1,\ldots,r)\geq \upsilon^{r}.
	\end{align}
\end{proof}

\begin{proof}[Proof of Lemma \ref{lem:Lip_Stated}]
	We can write
	$$
	\int_0^{\infty} \phi (y) \, P_{\theta}(x,\dd y) = \int_{-x}^{\infty} \phi (x+z) \, \zeta_{\theta}(z) \, \dd z + \phi(0) \cdot  P_{\theta}(x, \{0\}). 
	$$
	First, for the regular part (first term on the r.h.s.) we have  
    \begin{align} 
	& & \int_{-x}^{\infty}\phi(x+z) \left(\zeta_\theta(z)
	- \zeta_{\theta'}(z) \right)\, \dd z  \nonumber \\
    &\leq& \int_{-x}^{\infty}(1+\beta V(x+z))|\theta-\theta'|
	\, C^{\prime \prime} e^{-\eta^{\prime \prime} |z|} \, \dd z \nonumber \\	
    &\leq& C^{\prime \prime} |\theta-\theta'| \int_{0}^{\infty}(1+\beta e^{\chi (x+z)}) \,e^{-\eta^{\prime \prime} z}  \, \dd z \nonumber \\
	&+& C^{\prime \prime} |\theta-\theta'| \int_{-\infty}^{0}(1+\beta e^{\chi (x+z)})\, e^{-\eta^{\prime \prime} |z|}  \, \dd z \nonumber \\	
	\label{eq:Lip_for_Absolute_Cont_Part_of_P}	
	&\leq& C^{\prime \prime} |\theta-\theta'| \left({\frac 1 {\eta^{\prime \prime}}}  + \beta e^{\chi x} {\frac 1 {\eta^{\prime \prime} - \chi}} \right) \nonumber \\
	&+& C^{\prime \prime} |\theta-\theta'| \left({\frac 1 {\eta^{\prime \prime}}}  + \beta e^{\chi x} {\frac 1 {\eta^{\prime \prime} + \chi}} \right)\nonumber \\
	&\leq& C^{\prime \prime \prime} |\theta-\theta'| (1+\beta V(x)), 
	\end{align}
	with some constant $C^{\prime \prime \prime}.$ On the other hand, for the atomic component (second term on the r.h.s.) we get  
	\begin{align}
	\label{eq:Lip_for_Atomic_Part}
	& \phi(0)  \, \int_{-\infty}^{-x} \left(\zeta_\theta(z)
	- \zeta_{\theta'}(z) \right)\, \dd z  
	\le \nonumber \\
	& \phi(0) \int_{-\infty}^{-x} |\theta-\theta'|
	\, C^{\prime \prime}  e^{-\eta^{\prime \prime} |z|} \, \dd z \le \nonumber \\
	 & \phi(0) \,\, C^{\prime \prime} |\theta-\theta'| \, {\frac 1 {\eta^{\prime \prime}}}.
	\end{align}
	Inequalities \eqref{eq:Lip_for_Absolute_Cont_Part_of_P} and \eqref{eq:Lip_for_Atomic_Part} imply the claim of the lemma. 
\end{proof}

\end{appendices}

\bibliography{poisson}

\end{document}